\documentclass[12pt]{amsart}
\usepackage{amssymb}

\usepackage{amscd, amssymb, amsfonts, amsmath, amstext, amsthm}

\newcommand{\ot}{\otimes}

\newtheorem{theorem}{Theorem}[section]
\newtheorem{proposition}[theorem]{Proposition}
\newtheorem{corollary}[theorem]{Corollary}
\newtheorem{lemma}[theorem]{Lemma}
\theoremstyle{definition}
\newtheorem{definition}[theorem]{Definition}
\theoremstyle{remark}
\newtheorem{remark}[theorem]{Remark}
\numberwithin{equation}{section}

\DeclareMathSymbol{\rtimes}         {\mathbin}{AMSb}{"6F}

\begin{document}
\title{On semisimple Hopf algebras of dimension $2^{m}$, II.}
\author{Yevgenia Kashina}
\address{Department of Mathematical Sciences\\
DePaul University\\
Chicago, IL, 60614}
\email{ykashina@depaul.edu}

\begin{abstract}
In this paper we classify, up to equivalence, all semisimple nontrivial
Hopf algebras of dimension $%
2^{2n+1}$ for $n\geq 2$ over an algebraically closed field of characteristic $%
0$ with the group of group-like elements isomorphic to $\mathbb{Z}_{2^{n}}\times
\mathbb{Z}_{2^{n}}$. Moreover we classify all such nonisomorphic Hopf algebras of dimension $32$ and show that they are not twist-equivalent to each other. More generally, given an abelian group  of order $2^{m-1}$ we give an upper bound for the number of nonisomorphic nontrivial Hopf algebras of dimension $2^{m}$ which have this group as their group of group-like elements.
\end{abstract}

\maketitle

\allowdisplaybreaks
\section{Introduction}
An important problem in the theory of semisimple Hopf algebras is to construct examples satisfying certain properties and in the second step classify all Hopf algebras with these properties. Several such classification results were obtained in recent years, in particular in the series of papers by Masuoka (see, for example, \cite{Ma2},  \cite{Ma02},  \cite{Ma03},  \cite{Ma5}) and by Natale (see, for example \cite{Na1}, \cite{Na2}, \cite{Na3}, \cite{Na4},\cite{Na5}, \cite{Na6})  as well as in a recent preprint \cite{Kr} by Krop.

In this paper we continue the project of classifying semisimple Hopf algebras of dimension $2^m$ over an algebraically closed field of characteristic $%
0$  started in the papers \cite{Ka1} and \cite {Ka2}. In \cite{Ka1} we have shown that there are exactly $16$ nontrivial (that is, noncommutative and noncocommutative)
semisimple Hopf algebras of dimension $16$. In \cite{Ka2} we have classified all nontrivial  semisimple
 Hopf algebras of dimension $%
2^{n+1}$ for $n\geq 4$  with the group of group-like elements isomorphic to $\mathbb{Z}_{2^{n-1}}\times
\mathbb{Z}_{2}$.

We now classify, up to equivalence, all semisimple nontrivial
Hopf algebras of dimension $%
2^{2n+1}$ for $n\geq 2$ over an algebraically closed field of characteristic $%
0$ with the group of group-like elements isomorphic to $\mathbb{Z}_{2^{n}}\times
\mathbb{Z}_{2^{n}}$. We then give the complete list of nonisomorphic nontrivial Hopf algebras of dimension $32$ with the group of group-like elements isomorphic to $\mathbb{Z}_{4}\times
\mathbb{Z}_{4}$. We also show that these Hopf algebras are not twist-equivalent to each other. The above-mentioned Hopf algebras are all obtained as \textit{abelian extensions} of $kL%
$ by $k^G$ where $L\cong \mathbb{Z}_{2}$ and $G\cong \mathbb{Z}_{2^{n}}\times \mathbb{Z}%
_{2^{n}}$. Our ability to describe all such extensions heavily depends
on the structure of $G$, and in particular on the number of order $2$
automorphisms of $G$ and on the structure of the Schur multiplier $M(G)$.

The paper is organized as follows: In Section \ref{prelim} we gather the necessary preliminaries, in particular we discuss the cohomological description of abelian extensions based on works of Hofstetter, Masuoka, and Mastnak.

In Section \ref{2m} we consider  a group homomorphism $\Psi$ from the group of equivalence classes of abelian extensions of $kL$ by $k^G$, ${\rm H}^{2}\left( kL,k^{G}\right)$, to the Schur multiplier $M\left( G\right)$ (for a fixed action of $L\cong\mathbb{Z}_{2}$ on an arbitrary group $G$ of order  $2^{m-1}$). We show that its kernel is isomorphic to ${\rm H}^{2}\left( L,\hat G\right)  $ and that its image lies inside certain subgroup $ M\left( G\right)^{-}$. This implies that $\left| {\rm H}^{2}\left( kL,k^{G}\right)\right|\leq\left| {\rm H}^{2}\left( L,\hat G\right) \right| \cdot \left| M\left( G\right)^{-} \right|$. Moreover we conclude that there are at most $N\left(\left| \mathbf{G}\right|\cdot \left| M\left( \mathbf{G}\right)\right|/2 -1\right)$ nonisomorphic nontrivial Hopf algebras of dimension $2^{m}$ with group of group-like elements isomorphic to an abelian group $\mathbf{G}$ of order $2^{m-1}$, where $N$ is the number of distinct conjugacy classes of elements of order $2$ in ${\rm Aut}(\mathbf{G})$.

In Section \ref{main} we describe, up to equivalence, all possible non-commutative semisimple
Hopf algebras of dimension $2^{2n+1}$ with group of group-like elements isomorphic to $%
\mathbb{Z}_{2^{n}}\times \mathbb{Z}_{2^{n}}$ for $n\geq 2$. Every Hopf algebra under consideration fits into an abelian extension of $kL$ by $k^G$. In Theorem \ref{cohomology} we show that for a fixed action of $L$ on $G$
$$
{\rm H}^{2}\left( kL,k^{G}\right) \cong
  {\rm Ker}  \Psi \times {\rm Im} \Psi \cong {\rm H}^{2}\left( L,\hat G\right)\times M(G)^{-}.
 $$

In Section \ref{32} we describe all semisimple Hopf algebras of dimension $32$ with $\mathbf{G}\left(
H\right) \cong \mathbb{Z}_{4}\times \mathbb{Z}_{4}$ and therefore prove the following theorem:

\textsc{Theorem \ref{iso4x4}.}
{\it There are exactly $16$ nontrivial non-isomorphic semisimple
Hopf algebras of dimension $32$ with group of group-like elements isomorphic to $%
\mathbb{Z}_{4}\times \mathbb{Z}_{4}$. Moreover, all of them can be distinguished by categorical invariants, such as their Grothendieck ring structure or their (higher) Frobenius-Schur indicators. Therefore, the corresponding representation categories are also different and these Hopf algebras are not twist-equivalent to each other.}

The data required to distinguish non-isomorphic Hopf algebras is obtained in
Section \ref{32} and is gathered in Table 1. $N_{m,j}$ stands for the number of
 $2$-dimensional irreducible representations of $H$ with $m$-th Frobenius-Schur indicator equal to $j$. Note that by \cite[Theorem 3.1]{LM}, $N_{2,j}=0$ unless
$j=-1,0,1$. For non-abelian groups of order $16$ we use the notation from \cite[Section 4]{Ka1}:
\begin{eqnarray*}
 G_1&=&\left\langle a, b  |a^8=b^2=1, bab=a^5\right\rangle\\
G_6 &=&\left\langle a, b , c |a^4=b^2=c^2=1, bab=ac\right\rangle
\end{eqnarray*}
\newpage
\begin{center}
\textbf{Table 1}
\end{center}
\bigskip \noindent 
\begin{tabular}{|l|l|l|l|l|l|l|}
\hline
& $ G\cap Z\left( H\right)$ & $\mathbf{G}\left( H^{\ast }\right)$ & $%
N_{2,0}$ & $N_{2,1}$ & $N_{2,-1}$& Notes  \\ \hline
\begin{tabular}{l}
$H\left(\rightharpoonup_1, \alpha^2 _{1,2}\right)$  \\
$ H\left(\rightharpoonup_1, \alpha^2_{1,1}\alpha^2_{1,2}\alpha^2_{2,2}\right)$%
\end{tabular}
&
\begin{tabular}{l}
$\mathbb{Z}_{2} \times \mathbb{Z}_{2}$%
\end{tabular}
&
\begin{tabular}{l}
$\left(\mathbb{Z}_{2} \right)^3$%
\end{tabular}
& $0$ &
\begin{tabular}{l}
$4$ \\
$0$%
\end{tabular}
&
\begin{tabular}{l}
$2$ \\
$6$%
\end{tabular}
& \\ \hline
\begin{tabular}{l}
$H\left(\rightharpoonup_2,  \alpha _{1,2}\right)$\\
$H\left( \rightharpoonup_2,   \alpha ^2_{1,2}\right)$\\
  $H\left(\rightharpoonup_2,    \alpha^2_{1,2}\alpha^2_{2,2}\right)$\\
   $H\left(\rightharpoonup_2,  \beta _1 , \alpha _{1,2} \right)$\\
   $H\left(\rightharpoonup_2,  \beta _1 , \alpha _{1,2}\alpha ^{2}_{2,2} \right)$\\
    $H\left(\rightharpoonup_2,  \beta _1 , \alpha^2 _{1,2}\right)$
\end{tabular}
&
\begin{tabular}{l}
$\mathbb{Z}_{2} \times \mathbb{Z}_{4}$%
\end{tabular}
&
\begin{tabular}{l}
$G_6$\\
$\left(\mathbb{Z}_{2} \right)^2 \times \mathbb{Z}_{4}$\\
$\left(\mathbb{Z}_{2} \right)^2 \times \mathbb{Z}_{4}$\\
$G_1$\\
$G_1$\\
$\mathbb{Z}_{2} \times \mathbb{Z}_{8}$
\end{tabular}
& $2$ &
\begin{tabular}{l}
$1$ \\
$2$\\
$0$\\
$2$\\
$0$\\
$2$
\end{tabular}
&
\begin{tabular}{l}
$1$ \\
$0$\\
$2$\\
$0$\\
$2$\\
$0$
\end{tabular}
&\\ \hline
\begin{tabular}{l}
$H\left( \rightharpoonup_3,   \alpha ^2_{1,2}\right)$\\
$H\left(\rightharpoonup_3,  \beta _1 , \alpha^2 _{1,2}\right)$%
\end{tabular}
&
\begin{tabular}{l}
$\mathbb{Z}_{2} \times \mathbb{Z}_{4}$%
\end{tabular}
&
\begin{tabular}{l}
$\left(\mathbb{Z}_{2} \right)^2 \times \mathbb{Z}_{4}$ \\
$\mathbb{Z}_{4} \times \mathbb{Z}_{4}$%
\end{tabular}
& $4$ &
\begin{tabular}{l}
$0$
\end{tabular}
&
\begin{tabular}{l}
$0$
\end{tabular}
& \\ \hline
\begin{tabular}{l}
$H\left( \rightharpoonup_4,   \alpha ^2_{1,2}\right)$\\
  $H\left(\rightharpoonup_4,   \alpha^2 _{1,2} \alpha^2 _{2,2}\right)$
\end{tabular}
&
\begin{tabular}{l}
$\mathbb{Z}_{2} \times \mathbb{Z}_{2}$%
\end{tabular}
&
\begin{tabular}{l}
$\left(\mathbb{Z}_{2} \right)^{3} $
\end{tabular}
& $4$&
\begin{tabular}{l}
$2$ \\
$0$%
\end{tabular}
&
\begin{tabular}{l}
$0$ \\
$2$%
\end{tabular}
& \\ \hline
\begin{tabular}{l}
$H\left( \rightharpoonup_5,   \alpha^{\prime} _{1,2}\right)$\\
$H\left(\rightharpoonup_5,  \left(\alpha^{\prime} _{1,2} \right)^2\right)$
\end{tabular}
&
\begin{tabular}{l}
$\mathbb{Z}_{4}$%
\end{tabular}
&
\begin{tabular}{l}
$\left(\mathbb{Z}_{2} \right)^{3} $
\end{tabular}
& $6$ &
\begin{tabular}{l}
$0$%
\end{tabular}
&
\begin{tabular}{l}
$0$
\end{tabular}
& \begin{tabular}{l}
$N_{8,2} =2$ \\
$N_{8,2} =6$
\end{tabular}
\\ \hline
\begin{tabular}{l}
$ H\left( \rightharpoonup_6,   \alpha _{1,1}  \alpha _{1,2} \alpha _{2,2}\right)$\\
$H\left(\rightharpoonup_6, \alpha^2 _{1,1}  \alpha^2 _{1,2} \alpha^2 _{2,2}\right)$
\end{tabular}
&
\begin{tabular}{l}
$\mathbb{Z}_{2}\times\mathbb{Z}_{4}$%
\end{tabular}
&
\begin{tabular}{l}
$\mathbb{Z}_{2}\times \mathbb{Z}_{4}$%
\end{tabular}
& $4$ &
\begin{tabular}{l}
$2$%
\end{tabular}
&
\begin{tabular}{l}
$0$
\end{tabular}
& \begin{tabular}{l}
$N_{4,2} =0$ \\
$N_{4,2} =6$
\end{tabular}\\ \hline
\end{tabular}

\bigskip

\section{Preliminaries.}\label{prelim}
Throughout this paper we will assume that $k$ is an
algebraically closed field of characteristic $0$ and that $H$ is a semisimple
Hopf algebra over $k$  with
comultiplication $\Delta: H \longrightarrow H \ot H$, defined via $h \mapsto
\sum h_1 \ot h_2$, counit $\varepsilon : H \longrightarrow k$, and antipode $S$.
We let ${\mathbf G}(H)$ denote the group of group-like elements of $H$. For a general reference on Hopf algebras, see \cite{M}.

\begin{definition}[\protect{\cite[Definition 1.4]{Ma4}}]
Consider the sequence
\begin{equation}
K\overset{\iota}{\hookrightarrow }H\overset{\pi }{\twoheadrightarrow }F \label{ext}
\end{equation}
of finite-dimensional Hopf algebras and Hopf algebra maps $\iota$ and $\pi $ where $\iota$ is injective and $\pi $ is surjective. This sequence is called an \textit{extension} of $F$ by $K$ if $%
K=\left\{ h\in H|\sum h_{1}\otimes \pi \left( h_{2}\right) =h\otimes
1_{F}\right\} $, where $K\subset H$ via $i$. We say that two extensions $K \hookrightarrow H \twoheadrightarrow F$ and $K \hookrightarrow H' \twoheadrightarrow F$ are \textit{equivalent} if there is a Hopf algebra isomorphism $f:H\rightarrow H'$ which induces the identity maps on $K$ and $F$.
\end{definition}

For a general reference on Hopf algebra extensions, see \cite{H} and \cite{Ma4}.

\begin{definition}
Let $K$ and $F$ be Hopf algebras equipped with a weak action $%
\rightharpoonup :F\otimes K\rightarrow K$, a weak coaction $\rho
:F\rightarrow F\otimes K$, a cocycle $\sigma :F\otimes F\rightarrow K$ and a
dual cocycle $\tau :F\rightarrow K\otimes K$. The \textit{bicrossed}
\textit{product} $K\#_{\sigma }^{\tau }F$ is a Hopf algebra with $K\otimes
F$ as an underlying vector space and the following multiplication and
comultiplication:
\begin{eqnarray*}
\left( x\#a\right) \left( y\#b\right) =\sum x\left( a_{1}\rightharpoonup
y\right) \sigma \left( a_{2},b_{1}\right) \#a_{3}b_{2} \\
\Delta \left( x\#a\right) =\sum x_{1}\tau _{i}\left( a_{1}\right) \#\left(
a_{2}\right) ^{0}\otimes x_{2}\tau ^{i}\left( a_{1}\right) \left(
a_{2}\right) ^{1}\#a_{3}
\end{eqnarray*}
where
\begin{eqnarray*}
\rho \left( a\right) =\sum a^{0}\otimes a^{1} \\
\tau \left( a\right) =\sum_{i}\tau _{i}\left( a\right) \otimes \tau
^{i}\left( a\right)
\end{eqnarray*}
For the necessary conditions on $\rightharpoonup $, $\rho $, $\sigma $, and $%
\tau $ see \cite[3.1]{A}.
\end{definition}

For the rest of the paper we assume that $K$ is commutative and $F$ is
cocommutative. Then the extension (\ref{ext}) is called {\it abelian}; since $k$ is algebraically closed, it follows by
\cite[2.3.1]{M} that $K \cong (kG)^*$ and $F \cong kL$, for two finite
groups $G$ and $L$. Let $\{p_g|g\in G\}$ be the dual basis for $(kG)^*$.

By \cite[Proposition 1.5]{Ma4} any abelian extension of  $F$ by $K$ is equivalent to an extension
\begin{equation*}
K\overset{\iota}{\hookrightarrow }K\#_{\sigma }^{\tau }F\overset{\pi }{\twoheadrightarrow }F
\end{equation*}
where for $g\in G$, $a\in L$
\begin{eqnarray*}
\iota (p_g)=p_g\# 1\\
\pi (p_g\# a)=\delta _{1,g}a
\end{eqnarray*}

 As in \cite[page 891]{KMM} the action $\rightharpoonup$ of
$L$ on $K$ induces an action of $L$ on $K^* = kG$ via $(l\rightharpoonup f)(k) :=
f(Sl\rightharpoonup k)$. Since $K$ is commutative and $kL$ is cocommutative, $K$ is
a $kL$-module algebra, and thus $L$ acts as automorphisms of $K$. Thus $L$
permutes the orthogonal idempotents $\{p_g\}$; it follows that the action of
$L$ on $kG$ is in fact an action of $L$ on $G$ itself, which we also denote
by $\rightharpoonup$. Then the action of $L$ on $(kG)^*$ is given by
\begin{equation*}
x\rightharpoonup p_g = p_{x\rightharpoonup g}.
\end{equation*}

Let us now  fix the action $%
\rightharpoonup$ and coaction $\rho$. Equivalence classes of abelian extensions form the abelian group ${\rm Opext} (F,K)\cong {\rm H}^{2}\left( F,K\right)= {\rm Z}^{2}\left( F,K\right)/{\rm B}^{2}\left( F,K\right)$ (see \cite[Section 3]{H}),  
where ${\rm Z}^{2}\left( F,K\right)$ and ${\rm B}^{2}\left( F,K\right)$ are defined in \cite[page 268]{H}. Let us now assume that coaction $\rho$ is trivial. In this case more explicit description of ${\rm Z}^{2}\left( F,K\right)$ and ${\rm B}^{2}\left( F,K\right)$ can be found in \cite[Section 2]{Mas2}:

For  $a,b\in L, g,h\in G$ write
\begin{eqnarray*}
\sigma (a,b) &=&\sum_{g\in G}\sigma (g;a,b) p_{g}
\\
\tau (a) &=&\sum_{g,h\in G}\tau (g,h;a)p_{g}\otimes p_{h}
\end{eqnarray*}%
where $\sigma (g;a,b),\tau (g,h;a)\in k ^{\times }$, as in \cite[page 170]{Ma4}.  Then the pair $\left( \sigma ,\tau \right)$ lies in ${\rm Z}^{2}\left( F,K\right)$ if and only if for all $a,b,c\in L, g,h,k\in G$
\begin{eqnarray*}
\sigma \left(a^{-1}\rightharpoonup g;b,c\right)\sigma \left(g;a,bc\right) &=&\sigma \left(g;a,b\right) \sigma \left(g;ab,c\right)\\
\sigma \left(1;a,b\right) =\sigma \left(g;1,b\right) &=&\sigma \left(g;a,1\right) =1 \quad \text{normalization condition for }\sigma \\
\tau \left( g,h;a\right) \tau \left( gh,k;a\right)  &=&\tau \left(
g,hk;a\right) \tau \left( h,k;a\right) \\
\tau \left( g,h;1\right)  =\tau \left( 1,h;a\right) &=& \tau \left( g, 1;a\right)=1  \quad \text{normalization condition for }\tau\\
\tau \left( g,h;ab\right) \sigma \left(gh; a,b\right) &=&\tau \left(
g,h;a\right) \tau \left( a^{-1}\rightharpoonup g,a^{-1}\rightharpoonup
h;b\right) \sigma \left(g;a,b\right) \sigma \left( h;a,b\right)
\end{eqnarray*}
In addition, ${\rm B}^{2}\left( F,K\right)$ consists of pairs $\left( \sigma ,\tau \right)\in {\rm Z}^{2}\left( kL,k^{G}\right) $ such that there exists $\gamma
: L \to K ^{\times }$ defined via
$$
\gamma (a) = \sum_{g\in G}\gamma (g;a) p_{g} \quad \text{for } a\in L
$$
satisfying
\begin{eqnarray*}
\gamma \left( g;1\right) &=&\gamma \left( 1;g\right) =1\\
\sigma\left(g; a,b\right)  &=&\gamma \left( g;a\right) \gamma
\left( a^{-1}\rightharpoonup g;b\right) \gamma \left( g;ab\right) ^{-1}
\\
\tau \left( g,h;a\right)  &=&\gamma \left( g;a\right) ^{-1}\gamma
\left( h;a\right) ^{-1}\gamma \left( gh;a\right)
\end{eqnarray*}
for all $a,b\in L, g,h\in G$.

In \cite[Section 1]{Mas}, Mastnak introduced  ${\rm H}^{2}_c\left( F,K\right)$, the subgroup of  ${\rm H}^{2}\left( F,K\right)$, generated by the cocycles with trivial multiplication part. That is,
$$
{\rm H}_{c}^{2}\left( F,K\right) ={\rm Z}_{c}^{2}\left(F,K\right)
/{\rm B}_{c}^{2}\left( F,K\right)
$$ where
\begin{eqnarray*}
{\rm Z}_{c}^{2}\left( F,K\right) &=&\left\{ \tau \right | ({\rm triv}, \tau ) \in {\rm Z}^{2}\left( F,K\right)\}\\
{\rm B}_{c}^{2}\left( F,K\right) &=&{\rm B}^{2}\left( F,K\right)\cap {\rm Z}_{c}^{2}\left( F,K\right)
\end{eqnarray*}

Two sufficient conditions for $
{\rm H}_{c}^{2}\left( kL,k^G\right) =
{\rm H}^{2}\left( kL,k^G\right) $ were given in  \cite[Theorem 4.4]{Mas}. In particular, when $k$ is algebraically closed and $L$ is cyclic, $
{\rm H}_{c}^{2}\left( kL,k^G\right) =
{\rm H}^{2}\left( kL,k^G\right) $.

In order to distinguish between two non-isomorphic Hopf algebras we will use their invariants, in particular Frobenius-Schur indicators and Grothendieck rings.

In \cite{LM}, Frobenius-Schur indicators were extended from groups to Hopf algebras as follows:
\begin{definition}[\cite{LM}]
Let $\chi$ be a character of $H$. Define the $m$-th Frobenius-Schur indicator of $\chi$ via
$$
\nu_m\left( \chi\right):=\chi \left( \Lambda^{[m]}\right) =\sum \chi \left( \Lambda_1\cdots  \Lambda_m\right)
$$
where $\Lambda$ is a (unique) integral in $H$ such that $\varepsilon \left( \Lambda\right)=1$ (such a $\Lambda$ exists by Maschke's theorem).
\end{definition}
\begin{definition}[\cite{S}]
Let $K_{0}\left( H\right) ^{+}$ denote the abelian semigroup of all
equivalence classes of representations of $H$ with the addition given by a
direct sum. Then its enveloping group $K_{0}\left( H\right) $ has the
structure of an ordered ring with involution $^{\ast }$ and is called the
Grothendieck ring.
\end{definition}

The structure of $K_{0}\left( H\right) $ was described in \cite{NR}. The
multiplication in this ring is defined as follows: Let $\left[ \pi _{1}%
\right] $ and $\left[ \pi _{2}\right] $ denote the classes of
representations equivalent to $\pi _{1}$ and $\pi _{2}$, then $\left[ \pi
_{1}\right] \bullet \left[ \pi _{2}\right] $ is the class of the
representation $\left( \pi _{1}\otimes \pi _{2}\right) \circ \Delta $; the
unit of this ring is the class $\left[ \varepsilon \right] $ and $\left[ \pi %
\right] ^{\ast }$ is the equivalence class of the dual representation $%
^{t}\left( \pi \circ S\right) $ defined by $\left\langle ^{t}\left( \pi
\circ S\left( h\right) \right) \left( f\right) ,v\right\rangle =\left\langle
f,\left( \pi \circ S\left( h\right) \right) \left( v\right) \right\rangle $.
The equivalence classes of irreducible representations of $H$ form a basis
of $K_{0}\left( H\right) $. Moreover, $\left[ \pi _{1}\right] ^{\ast } = %
\left[ \pi _{2}\right] $ if and only if $\left[ \varepsilon
\right] $ is a summand in the basis decomposition of $\left[
\pi_{1}\right] \bullet \left[ \pi _{2}\right] $ (necessarily of
multiplicity $1$). For simplicity of notation we will write $\pi
$ instead of $\left[ \pi \right] $ for elements of $K_{0}\left(
H\right) $. In this paper, all irreducible representations of Hopf algebras under consideration will be $1$- or $2$-dimensional. We will not completely describe the structure of Grothendieck rings; it would be enough for our purposes to compute $\pi ^2$ for each $2$-dimensional irreducible representation $\pi$ of $H$.

\section{Hopf algebras of dimension $2^{m}$ with commutative Hopf subalgebra of dimension $2^{m-1}$.}\label{2m}

Let $H$ be a semisimple Hopf algebra of dimension $2^{m}$ over an algebraically closed field $k$ of characteristic $0$ and assume that it has a commutative Hopf subalgebra $K$ of dimension   $2^{m-1}$. Then, by the works of Masuoka (see \cite{Ma2}
and \cite{Ma5}) or by \cite[Proposition 2.4]{Ka2},
$H$ is equivalent to the bicrossed product $%
K\#_{\sigma }^{\tau }F$ of $K=k^G=\left( kG \right) ^{\ast }$ where $G $ is a group of order $2^{m-1}$ and $%
F=kL=k\left\langle t\right\rangle \cong k\mathbb{Z}_{2}$ with an action $%
\rightharpoonup :F\otimes K\rightarrow K$, trivial coaction, a cocycle $%
\sigma :F\otimes F\rightarrow K$ and a dual cocycle $\tau
:F\rightarrow K\otimes K$ (that is,  $\left( \sigma ,\tau \right) \in {\rm Z}^{2}\left( kL,k^{G}\right) $), and the action by $t$ being a
Hopf algebra automorphism of $K$.

Since $\sigma$ and $\tau$ are normalized they are completely determined by $\sigma (t,t)$ and $\tau (t)$. We can write $\sigma (t,t)$ and $\tau (t)$ in terms of
the basis for $(kG)^{\ast }$:
\begin{eqnarray}
\sigma (t,t) &=&\sum_{g\in G}\sigma _{t}(g)p_{g}  \label{cocyclecomp}
\\
\tau (t) &=&\sum_{g,h\in G}\tau _{t}(g,h)p_{g}\otimes p_{h}  \notag
\end{eqnarray}%
where $\sigma _{t}(g)=\sigma (g;t,t),\tau _{t}(g,h)=\tau (g,h;t)\in k$. Since $\left( \sigma ,\tau \right) \in {\rm Z}^{2}\left( kL,k^{G}\right) $,  $\sigma _{t}$ and $\tau _{t}$ satisfy
the following properties for any $ g,h\in G$:
\begin{eqnarray}
\sigma _{t}\left( 1\right)  &=&1  \label{sigma1} \\
\tau _{t}\left( 1,h\right)  &=&\tau _{t}\left(
g,1\right) =1   \\
\sigma _{t}\left( t\rightharpoonup g\right)  &=&\sigma _{t}\left(
g\right)   \label{sigma2} \\
\tau _{t}\left( g,h\right) \tau _{t}\left( gh,k\right)  &=&\tau _{t}\left(
g,hk\right) \tau _{t}\left( h,k\right)   \\
\sigma _{t}\left( gh\right)  &=&\tau _{t}\left( g,h\right) \tau _{t}\left(
t\rightharpoonup g,t\rightharpoonup h\right) \sigma _{t}\left(g\right)
\sigma _{t}\left( h\right)   \label{sigma3}
\end{eqnarray}

By the definition of ${\rm B}^{2}\left( kL,k^{G}\right)$, we get the following lemma (note that $\gamma$ from the lemma corresponds to $\gamma ( \_\ , t)$ from the definition of ${\rm B}^{2}\left( kL,k^{G}\right)$ in Section \ref{prelim}):
\begin{lemma}
Let \label{equiv2}$\gamma \in \left( k^{G}\right) ^{\times }$ and $\gamma
\left( 1\right) =1$. Then the extension $H=\left( kG\right) ^{\ast
}\#_{\sigma }^{\tau }kL$ is equivalent to the extension $H_{\gamma }=\left(
kG\right) ^{\ast }\#_{^{\gamma }\sigma }^{\tau ^{\gamma ^{-1}}}kL$, where
\begin{eqnarray*}
\left( ^{\gamma }\sigma \right) _{t}(g)=\gamma \left( g\right) \gamma
\left( t\rightharpoonup g\right) \sigma _{t}(g) \\
\left( \tau ^{\gamma ^{-1}}\right) _{t}\left( g,h\right) =\gamma
\left( gh\right) \left( \gamma \left( g\right) \gamma \left(
h\right) \right) ^{-1}\tau _{t}\left( g,h\right)
\end{eqnarray*}
\end{lemma}

\begin{remark}\label{trivial}
Since $k$ is algebraically closed and $L$ is cyclic, by \cite[Chapter IV, Theorem 7.1]{Mac},
${\rm H}^{2}\left( L,\left( k^{G}\right) ^{\times }\right) =0$ and therefore $\sigma \in {\rm B}^{2}\left( L,\left( k^{G}\right) ^{\times }\right)$. Then $\left(\sigma \right) _{t}(g)=\gamma \left( g\right) \gamma
\left( t\rightharpoonup g\right)$ for some $\gamma \in \left( k^{G}\right) ^{\times }$ with $\gamma
\left( 1\right) =1$. Thus $\left( kG\right) ^{\ast }\#_{\sigma }^{\tau }kL$ is equivalent to
$\left( kG\right) ^{\ast }\#^{\tau ^{\gamma }}kL$. Therefore in this case Lemma \ref{equiv2} implies the result of \cite[Theorem 4.4 (1)]{Mas} that ${\rm H}^{2}\left( kL,k^{G}\right) =
{\rm H}^{2}_{c}\left( kL,k^{G}\right) $.
\end{remark}

\begin{definition} We will say that two actions of $L$ on $G$, $\rightharpoonup _1$ and $\rightharpoonup _2$, are {\it similar} if there is a group
automorphism $f:G\longrightarrow G$  such that  $f ( t\rightharpoonup_1 g) =t\rightharpoonup_2 f (  g) $ for all $g\in G$.
\end{definition}
\begin{remark}
Since $\left| L\right| =2$, the number of non-similar non-trivial actions of $L$ on $G$ equals the number of distinct conjugacy classes of elements of order $2$ in ${\rm Aut}(\mathbf{G})$.
\end{remark}
The following result is contained implicitly in the works by Masuoka (see, for example \cite{Ma2} or \cite{Ma03}). We give the proof for completeness.
\begin{lemma}
Let  $
F=k\left\langle t\right\rangle \cong k\mathbb{Z}_{2}$ and $K=k^G$ for a finite group $G $. Consider nontrivial Hopf algebras  $H_1=K\#_{\sigma_1 }^{\tau_1 }F$  and $H_2=K\#_{\sigma_2 }^{\tau_2 }F$. Assume that the coactions are trivial and the corresponding actions, $\rightharpoonup _1$ and $\rightharpoonup _2$, by $t$ are group automorphisms of $G$. Then
\begin{enumerate}
\item If $\rightharpoonup _1$ and $\rightharpoonup _2$ are similar with $f ( t\rightharpoonup_1 g) =t\rightharpoonup_2 f (  g) $ , $\left(\sigma _1\right) _{t}\left( g\right) = \left(\sigma_2\right) _{t}\left( f(g)\right)$, and $\left(\tau _1\right) _{t}\left( g,h\right) = \left(\tau _2\right) _{t}\left( f(g),f(h)\right)$ for all $g,h\in G$ then $H_1$ and $H_2$ are isomorphic.
\item If $G$ is abelian and $H_1$ and $H_2$ are isomorphic then $\rightharpoonup _1$ and $\rightharpoonup _2$ are similar.
\end{enumerate}
\end{lemma}
\begin{proof}
\begin{enumerate}
\item Define
\begin{eqnarray*}
\mathcal{F} :
k^G\#_{\sigma_1 }^{\tau_1 }kL &\longrightarrow &
k^G\#_{\sigma_2 }^{\tau _2}kL \qquad \text{  by}\\
\mathcal{F}\left( p_{g}\#t \right)
&=&
p_{f(g)}\# t
\end{eqnarray*}
It is straightforward to check that $\mathcal{F}$ is a Hopf algebra isomorphism.
\item
Since $G$ is abelian, $K=k^G=k\hat G$ and therefore $\hat G \subseteq \mathbf{G}\left( H_1\right)=\mathbf{G}\left( H_2\right)$. Since $[H_1:K]=[H_2:K]=\dim F =2$ and $H_1$ and $H_2$ are nontrivial, $\hat G = \mathbf{G}\left( H_1\right)=\mathbf{G}\left( H_2\right)$ and $K= k\widehat{\mathbf{G}\left( H_1\right) }=k\widehat{\mathbf{G}\left( H_2\right) }$. Group-like elements are sent to group-like elements by Hopf algebra maps and therefore $K$ is preserved under any Hopf isomorphism ${\mathcal F}: H\longrightarrow H$. Then ${\mathcal F}$ restricts to an automorphism ${\mathcal F}_{k^G}: k^G\longrightarrow k^G$ and induces a group  automorphism $f:G\longrightarrow G$ via ${\mathcal F} (p_g) = p_{f^{-1}(g)}$.  Denote $1 \# t= \bar t$ and $p_g \# t= p_g\bar t$. Then
\begin{eqnarray*}
{\mathcal F}\left( \bar t p_g \right)&=&{\mathcal F}\left( p_{t\rightharpoonup_1 g} \bar t\right) ={\mathcal F}\left( p_{t\rightharpoonup_1 g}\right){\mathcal F}\left(  \bar t\right) = p_{f^{-1}(t\rightharpoonup_1 g)}{\mathcal F}\left(  \bar t\right)\\
{\mathcal F}\left( \bar t  \right){\mathcal F}\left(  p_g \right) &=& {\mathcal F}\left( \bar t  \right)p_{f^{-1}( g)}=p_{t\rightharpoonup_2 f^{-1}( g)}{\mathcal F}\left(  \bar t\right)
\end{eqnarray*}
Thus $p_{f^{-1}(t\rightharpoonup_1 g)}=p_{t\rightharpoonup_2 f^{-1}( g)}$ and therefore $f^{-1}(t\rightharpoonup_1 g)=t\rightharpoonup_2 f^{-1}( g)$.
\end{enumerate}
\end{proof}

Let us now fix an action of $L$ on $G$. By \cite[Theorem 4.4 (1)]{Mas} (or Remark \ref{trivial}), every extension is equivalent to an extension with a trivial cocycle. We will identify the equivalence class of $\left( kG\right) ^{\ast }\#^{\tau }kL$ with $\left[ \left( {\rm triv}, \tau \right) \right]=\left[
\tau \right]\in {\rm H}^{2}_{c}\left( kL,k^{G}\right)$, where $\tau \in {\rm Z}_{c}^{2}\left( kL,k^{G}\right) $ satisfies
\begin{eqnarray}
\tau _{t}\left( 1,h\right) &=&\tau _{t}\left(
g,1\right) =1  \label{coc1} \\
\tau _{t}\left( g,h\right) \tau _{t}\left( gh,k\right)  &=&\tau _{t}\left(
g,hk\right) \tau _{t}\left( h,k\right)   \label{coc2} \\
1 &=&\tau _{t}\left( g,h\right) \tau _{t}\left( t\rightharpoonup
g,t\rightharpoonup h\right)   \label{fix}
\end{eqnarray}
Thus our extension is completely
 determined
by the cocycle $\tau _{t} \in {\rm Z}^{2}\left( G,k^{\times }\right) $ which satisfies property (\ref{fix}).

$L$ acts on ${\rm Z}^{2}\left( G,k^{\times }\right) $ via
$$(t\rightharpoonup \alpha )\left(
g,h\right) =\alpha \left( t\rightharpoonup g,t\rightharpoonup
h\right)
$$
This action induces an action of $L$ on
${\rm H}_{c}^{2}\left( kL,k^{G}\right)$.

Consider the Schur multiplier
$$
M\left( G\right) ={\rm H}^{2}\left( G,k^{\times }\right).
$$
Then $L$ acts on $M\left( G\right) $ via
$$
t\rightharpoonup [\alpha ]
=[ t\rightharpoonup \alpha ]
$$
Consider the subgroup  $M\left( G\right)^{-}$ of $M\left( G\right)$ defined by
$$
M\left( G\right)^{-} =\left\{ [\alpha ] \in M\left(
G\right) |t\rightharpoonup [\alpha ]= [\alpha ]^{-1}\right\}=\left\{ [\alpha ] \in M\left(
G\right) |[\alpha ]\left( t\rightharpoonup [\alpha ]\right)= 1\right\}.
$$
\begin{theorem}\label{psi}
Let $G $ be a group of order $2^{m-1}$ and $%
L=\left\langle t\right\rangle \cong \mathbb{Z}_{2}$ as in the beginning of this section. Consider the map
\begin{eqnarray*}
\Psi :{\rm H}^{2}\left( kL,k^{G}\right)={\rm H}_{c}^{2}\left( kL,k^{G}\right) &\longrightarrow & M\left( G\right)
 \\
\Psi \left( \left[ \left( {\rm triv}, \tau \right) \right] \right) =\Psi \left( \left[ \tau \right] \right) &=&\left[ \tau _{t}\right]
\end{eqnarray*}
Then
\begin{enumerate}
\item  $\Psi$ is a well-defined group homomorphism.
\item ${\rm Im}  \Psi \subseteq M\left( G\right)^{-}$.
\item  ${\rm Ker}  \Psi\cong \hat G ^L/ N_t\left(\hat G\right)\cong {\rm H}^{2}\left( L,\hat G\right)  $.
\end{enumerate}
\end{theorem}
\begin{proof}
Consider $\tau \in {\rm Z}_{c}^{2}\left( kL,k^{G}\right) $. Then, by $\left( \ref{coc1}\right) $ and $%
\left( \ref{coc2}\right) $,  $\tau _{t}\in
{\rm Z}^{2}\left( G,k^{\times }\right) $. If $\tau \in {\rm B}_{c}^{2}\left( kL,k^{G}\right) $ then $\tau _{t}\left(
g,h\right) =\gamma _{t}\left( g\right) ^{-1}\gamma _{t}\left( h\right)
^{-1}\gamma _{t}\left( gh\right) $ and therefore $\tau _{t}\in {\rm B}^{2}\left(
G,k^{\times }\right) $. Thus $\Psi$ is  well-defined.

$\Psi $ is a homomorphism since $\left( \tau \tau ^{\prime
}\right) _{t}=\tau _{t}\tau _{t}^{\prime }$.

If $\left[ \tau \right] \in {\rm H}_{c}^{2}\left( kL,k^{G}\right) $ then $%
\tau _{t}\left( g,h\right) \tau _{t}\left( t\rightharpoonup
g,t\rightharpoonup h\right) =1$ by $\left( \ref{fix}\right) $ and therefore
\begin{equation*}
\left[ \tau _{t}\right]\left(t\rightharpoonup \left[ \tau _{t}\right]\right) =\left[ t\rightharpoonup \tau _{t}%
\right] =1
\end{equation*}

Assume now that $\left[ \left( {\rm triv}, \tau \right) \right]=\left[
\tau \right]\in {\rm Ker} \Psi$. Then $\tau _{t}\in {\rm B}^{2}\left(
G,k^{\times }\right) $ and therefore $\tau _{t}\left( g,h\right)
=\gamma \left( g\right) \gamma \left( h\right) \gamma \left(
gh\right) ^{-1}$ for some $\gamma \in \left( k^{G}\right) ^{\times }$ with  $\gamma (1) =1$. Then $\left[ \left( {\rm triv},\tau
\right) \right] =\left[ \left( \sigma ,{\rm triv}\right) \right]
\in {\rm H}^{2}\left( kL,k^{G}\right)$ where $\sigma =\! ^{\gamma}{\rm triv}$, that is,
$$\sigma _{t}\left( g\right) =\gamma \left(
g\right) \gamma \left( t\rightharpoonup g\right).$$
 Formulas (\ref{sigma1}), (\ref{sigma2}), and (\ref{sigma3}) imply that $\sigma$ satisfies the following properties
\begin{eqnarray*}
\sigma _{t}\left( 1\right)  &=&1\\
\sigma _{t}\left( t\rightharpoonup g\right)
&=&\sigma _{t}\left( g\right) \\
\sigma _{t}\left( gh\right)  &=&\sigma _{t}\left(
g\right) \sigma _{t}\left( h\right)
\end{eqnarray*}
Therefore $\sigma _t \in \hat G ^L$ the subgroup of $\hat G$, fixed by $L$. Then
$\sigma = \sigma^{\beta}$ for some $\beta \in \hat G ^L$, where
$\left( \sigma^{\beta}\right) _{t}= \beta $. Define a map
\begin{eqnarray*}
\phi: \hat G ^L&\longrightarrow &{\rm Ker} \Psi \qquad \text{via }
\phi(\beta) = \left[ \left( \sigma^{\beta} ,{\rm
triv}\right) \right]
\end{eqnarray*}
Since $\beta \in \hat G ^L$, $\left( \sigma^{\beta} ,{\rm
triv}\right) $ satisfies properties (\ref{sigma1})--(\ref{sigma3}) and therefore $\left( \sigma^{\beta} ,{\rm
triv}\right) \in  Z^{2}\left( kL,k^{G}\right) $. Since
${\rm H}^{2}\left( L,\left( k^{G}\right) ^{\times }\right) =0$,  $\left(\sigma^{\beta} \right) _{t}(g)=\gamma \left( g\right) \gamma
\left( t\rightharpoonup g\right)$ for some $\gamma \in \left( k^{G}\right) ^{\times }$ with $\gamma
\left( 1\right) =1$. Therefore  $\left[ \left( \sigma^{\beta} ,{\rm triv}\right) \right] = \left[ \left( {\rm triv}, {\rm triv}^{\gamma}\right) \right]$ and
$$\Psi \left( \left[ \left( \sigma^{\beta} ,{\rm triv}\right) \right] \right) =\Psi \left( \left[ \left( {\rm triv}, {\rm triv}^{\gamma} \right) \right] \right) =\left[ {\rm triv}^{\gamma}  \right]=1$$
Thus  $\phi$ is well-defined. Since $\sigma^{\beta_1 \beta_2}=\sigma^{\beta_1}\sigma^{\beta_2}$, $\phi$ is a homomorphism.

Since we have already shown that for every $\left[ \left( {\rm triv}, \tau \right) \right]\in {\rm Ker} \Psi$ there exists $\beta \in \hat G ^L$ such that $\left[ \left( {\rm triv}, \tau \right) \right] =
 \left[ \left( \sigma^{\beta}
 ,{\rm triv}\right) \right] $, $\phi$ is surjective. Moreover
\begin{eqnarray*}
 {\rm Ker}\phi  &=&\{\beta \in  \hat G ^L | \beta = \gamma
 \left( t\rightharpoonup \gamma \right) \text{ and } \gamma \left( g\right) \gamma \left( h\right) \gamma \left(
gh\right) ^{-1} =1\}\\
 &=& \{ \gamma
 \left( t\rightharpoonup \gamma \right) | \gamma \in \hat G\}=N_t\left(\hat G\right).
 \end{eqnarray*}

Thus ${\rm Ker}  \Psi\cong \hat G ^L/ N_t\left(\hat G\right)\cong {\rm H}^{2}\left( L,\hat G\right)  $.
\end{proof}
\begin{remark}
Since $t^2=1$, the map $\Psi$ is the coconnecting homomorphism defined in \cite[Section 6]{Mas}.
\end{remark}

\begin{corollary}
Let $G $ be a group of order $2^{m-1}$ and $%
L$ be a group of order $2$. For any given action $\rightharpoonup $ of $L$ on $G$, $\left| {\rm H}^{2}\left( kL,k^{G}\right)\right|\leq\left| {\rm H}^{2}\left( L,\hat G\right) \right| \cdot \left| M\left( G\right)^{-} \right|$.
\end{corollary}
\begin{corollary}
Let $\mathbf{G}$ be an abelian group of order $2^{m-1}$ and let $N$ be the number of distinct conjugacy classes of elements of order $2$ in ${\rm Aut}(\mathbf{G})$. Then there are at most $N\left(\left| \mathbf{G}\right|\cdot \left| M\left( \mathbf{G}\right)\right|/2 -1\right)$ nonisomorphic Hopf algebras of dimension $2^{m}$ with group of group-like elements isomorphic to $\mathbf{G}$.
\begin{proof}
Every non-trivial Hopf algebra under consideration is equivalent to $H=k^G\#^{\tau }kL$ where $G =\widehat{\mathbf{G}}$ and $\left| L\right| =2$. Therefore the action of $L$ on $G$ is nontrivial and corresponds to an automorphism of $G$ of order $2$. Moreover, $\hat G ^L$ is a proper subgroup of $\hat G $ and therefore
$$
\left| {\rm H}^{2}\left( L,\hat G\right) \right| \leq \left| \hat G ^L \right|\leq \left| \hat G  \right|/2 = \left| \mathbf{G}\right|/2
$$
\end{proof}
\end{corollary}


\section{Hopf algebras of dimension $2^{2n+1}$ with $\mathbf{G}\left(
H\right) =\mathbb{Z}_{2^{n}}\times \mathbb{Z}_{2^{n}}$}

\label{main} In this section we describe, up to equivalence, all possible non-commutative semisimple
Hopf algebras of dimension $2^{2n+1}$ with group of group-like elements isomorphic to $%
\mathbb{Z}_{2^{n}}\times \mathbb{Z}_{2^{n}}$ for $n\geq 2$.

Let $H$ be a nontrivial semisimple Hopf algebra of dimension $2^{2n+1}$ with $%
\mathbf{G}\left( H\right) \cong \mathbb{Z}_{2^{n}}\times \mathbb{Z}_{2^{n}}$.
Then, as in Section \ref{2m},
$H$ is equivalent to the bicrossed product $%
K\#_{\sigma }^{\tau }F$ of $K=k\mathbf{G}\left( H\right)=\left( kG \right) ^{\ast }$ where $G =\widehat{\mathbf{G}\left( H\right) }\cong \mathbb{Z}_{2^{2n}}\times \mathbb{Z}_{2^{2n}}$ and $%
F=kL=k\left\langle t\right\rangle \cong k\mathbb{Z}_{2}$ with an action $%
\rightharpoonup :F\otimes K\rightarrow K$ by Hopf algebra automorphisms, trivial coaction, a cocycle $%
\sigma :F\otimes F\rightarrow K$, and a dual cocycle $\tau
:F\rightarrow K\otimes K$.

Let $G=\left\langle x\right\rangle \times \left\langle
y\right\rangle \cong\mathbb{Z}_{2^{n}}\times \mathbb{Z}_{2^{n}}$. Then $\hat G = \mathbf{G}\left( H\right) =\left\langle \beta _1 \right\rangle \times \left\langle \beta _2  \right\rangle \cong \mathbb{Z}_{2^{n}}\times \mathbb{Z}_{2^{n}}$ where
  $$
    \beta_{1}\left( x^{a}y^{b}\right) =\xi^{ a}\quad \beta_{2}\left( x^{a}y^{b}\right) =\xi^{ b}
    $$
and $\xi$ is a primitive $2^n$-th root of unity.

Define
$$
\alpha _{p,q}\left( x^{a_{1}}y^{a_{2}},x^{b_{1}}y^{b_{2}}\right) =\xi
^{a_{p}b_{q}}
$$
Then $M\left( G\right) = \left\{ \left[ \alpha^k _{1,2}\right] |k\in\mathbb{Z}_{2^n}
\right\}   = \left\langle \left[ \alpha _{1,2}\right]
\right\rangle\cong \mathbb{Z}_{2^n}$ (see \cite[Corollary 2.2.12]{Kar}).

Define $\tau^k _{p,q}$ via
$$
\left(\tau^k _{p,q}\right) _t = \alpha^k _{p,q}.
$$
Then $\tau^k _{p,q}$ is a dual cocycle.

\begin{proposition}\label{12_21}
\begin{enumerate}
\item In $M\left( G\right)$
\begin{eqnarray*}
\left[ \alpha _{2,1}\right] &=&\left[ \alpha ^{-1}_{1,2}\right]\\
\left[ \alpha _{j,j}\right] &=&1
\end{eqnarray*}
where $j=1,2$.
\item If  $t\rightharpoonup x = xy^{r}$ and
$t\rightharpoonup y = y^s$ then in $
 {\rm H}^{2}\left( kL,k^{G}\right)$, $[(\sigma^{\beta_1} , {\rm triv})] =  [( {\rm triv}, \tau '  )] $  where $\tau '_t =\alpha '$ for $\alpha '$ defined via
$$
\alpha ' \left( x^{a_{1}}y^{a_{2}},x^{b_{1}}y^{b_{2}}\right) =\left( -1\right)
^{k}
$$
with $k= 0$ if $0\leq a_1+b_1 < 2^{n}$ and $k = 1$ if $2^{n}\leq a_1+b_1 < 2^{n+1}$.
\item If  $t\rightharpoonup x = x^{q}y^{2^{n-1}r}$ and
$t\rightharpoonup y = y^{-1}$ then  in $
 {\rm H}^{2}\left( kL,k^{G}\right)$,
$$
[(\sigma^{\beta^{2^{n-1}} _2} , {\rm triv})] =  [( {\rm triv}, \tau ^{2^{n-1}}_{2,2} )].
$$

\item If  $t\rightharpoonup x = x^{-1}$ and
$t\rightharpoonup y = y^{s}$ then  in $
 {\rm H}^{2}\left( kL,k^{G}\right)$,
$$
[(\sigma^{\beta^{2^{n-1}} _1} , {\rm triv})] =  [( {\rm triv}, \tau ^{2^{n-1}}_{1,1} )].
$$
\end{enumerate}
\end{proposition}
\begin{proof}
Let $j\in\{1,2\}$ and let $w$ be a primitive $2^{n+1}$-th root of $1$, such that $\omega^2 =-\xi$.
Define $\mu, \gamma _j, \eta
\in \left( k^{G}\right) ^{\times }$  via
\begin{eqnarray*}
\mu \left( x^{a_{1}}y^{a_{2}}\right)
&=&\xi^{a_{1}a_{2}}\\
\gamma _j
\left( x^{a_{1}}y^{a_{2}}\right) &=&\omega^{a_{j}^{2}}\\
\eta \left( x^{a_{1}}y^{a_{2}}\right) &=&\omega^{a_{1}}
\end{eqnarray*}
Then
\begin{enumerate}
\item
\begin{eqnarray*}
\alpha _{2,1}\left( x^{a_{1}}y^{a_{2}},x^{b_{1}}y^{b_{2}}\right)
\mu \left( x^{a_{1}}y^{a_{2}}\right) \mu \left( x^{b_{1}}y^{b_{2}}\right) \mu \left( x^{a_{1}}y^{a_{2}}x^{b_{1}}y^{b_{2}}\right)
^{-1} &=&\xi ^{a_{2}b_{1}}\xi^{a_{1}a_{2}}\xi^{b_{1}b_{2}}\xi^{-(a_{1}+b_{1})(a_{2}+b_{2})}\\
&=&\xi
^{-a_{1}b_{2}}=\alpha ^{-1}_{1,2}\left( x^{a_{1}}y^{a_{2}},x^{b_{1}}y^{b_{2}}\right)\\
\gamma _j\left( x^{a_{1}}y^{a_{2}}\right)
\gamma _j \left( x^{b_{1}}y^{b_{2}}\right)\gamma _j ^{-1}\left( x^{a_{1}}y^{a_{2}}x^{b_{1}}y^{b_{2}}\right)
&=&\omega^{a_{j}^{2}}\omega^{b_{j}^{2}}\omega^{-\left( a_{j}+b_{j} -2^n k\right) ^{2}}\\
&=&\xi
^{a_{j}b_{j}}=\alpha _{j,j}\left( x^{a_{1}}y^{a_{2}},x^{b_{1}}y^{b_{2}}\right)
\end{eqnarray*}
where $k= 0$ if $0\leq a_j+b_j <2^n$ and $k = 1$ if $2^n\leq a_j+b_j <2^{n+1}$.

\item Assume now that $t\rightharpoonup x = xy^{r}$ and
$t\rightharpoonup y = y^s$. Then $\beta_1 \in \hat G ^L$ and $[(\sigma^{\beta_1} , {\rm triv})] \in
 {\rm H}^{2}\left( kL,k^{G}\right)$. We will show that $[(\sigma^{\beta_1} , {\rm triv})] =  [( {\rm triv}, \tau '  )] $  where $\tau '_t =\alpha '$ with the notation $\sigma =\sigma^{\beta _1}$ and  $\tau= {\rm triv}$:
\begin{eqnarray*}
\left( ^{\eta }\sigma \right) _{t}\left( x^{a_{1}}y^{a_{2}}\right)
&=&\eta \left( x^{a_{1}}y^{a_{2}}\right) \eta \left( t\rightharpoonup
x^{a_{1}}y^{a_{2}}\right) \sigma _{t}\left( x^{a_{1}}y^{a_{2}}\right)\\
&=&\eta \left( x^{a_{1}}y^{a_{2}}\right) \eta \left(
x^{a_{1}}y^{ra_1+sa_{2}}\right)\beta _1\left( x^{a_{1}}y^{a_{2}}\right)\\
&=&\omega^{a_{1}}\omega^{a_{1}} \xi ^{a_{1}}=1 \\
\left( \tau ^{\eta ^{-1}}\right) _{t}\left( x^{a_{1}}y^{a_{2}},x^{b_{1}}y^{b_{2}}\right)  &=&\eta \left( x^{a_{1}}y^{a_{2}}x^{b_{1}}y^{b_{2}}\right) \eta ^{-1}\left( x^{a_{1}}y^{a_{2}}\right)
\eta ^{-1}\left( x^{b_{1}}y^{b_{2}}\right)  \\
&=& \omega^{a_{1}+b_1 -2^n k}\omega^{-a_{1}}\omega^{-b_{1}}=\left( -1\right)
^{k} \\
&=& \alpha ' \left( x^{a_{1}}y^{a_{2}},x^{b_{1}}y^{b_{2}}\right)
\end{eqnarray*}
where $k= 0$ if $0\leq a_1+b_1 < 2^{n}$ and $k = 1$ if $2^{n}\leq a_1+b_1 < 2^{n+1}$.

\item Assume now that  $t\rightharpoonup x = x^{q}y^{2^{n-1}r}$ and
$t\rightharpoonup y = y^{-1}$. Then $\beta^{2^{n-1}} _2 \in \hat G ^L$ and $[(\sigma^{\beta^{2^{n-1}} _2} , {\rm triv})] \in
 {\rm H}^{2}\left( kL,k^{G}\right)$. We will show that $[(\sigma^{\beta^{2^{n-1}} _2} , {\rm triv})] =  [( {\rm triv}, \tau ^{2^{n-1}}_{2,2} )] $.  Denote $\sigma =\sigma^{\beta^{2^{n-1}} _2}$, $\tau= {\rm triv}$ and $\gamma =\gamma^{2^{n-1}} _2$. Then
\begin{eqnarray*}
\left( ^{\gamma }\sigma \right) _{t}\left( x^{a_{1}}y^{a_{2}}\right)
&=&\gamma \left( x^{a_{1}}y^{a_{2}}\right) \gamma \left( t\rightharpoonup
x^{a_{1}}y^{a_{2}}\right) \sigma _{t}\left( x^{a_{1}}y^{a_{2}}\right)\\
&=&\gamma \left( x^{a_{1}}y^{a_{2}}\right) \gamma \left( x^{qa_{1}}y^{2^{n-1}ra_{1}-a_{2}}\right)
\beta^{2^{n-1}} _2\left( x^{a_{1}}y^{a_{2}}\right)\\
&=&i^{a_{2}^{2}}i^{(2^{n-1}ra_{1}-a_{2})^{2}}\left( -1\right) ^{a_{2}}=1 \\
\left( \tau ^{\gamma ^{-1}}\right) _{t}\left( x^{a_{1}}y^{a_{2}},x^{b_{1}}y^{b_{2}}\right)
&=&\gamma\left( x^{a_{1}}y^{a_{2}}x^{b_{1}}y^{b_{2}}\right) \gamma^{-1}\left(
x^{a_{1}}y^{a_{2}}\right)
\gamma ^{-1}\left( x^{b_{1}}y^{b_{2}}\right)  \\
&=&i^{\left( a_{2}+b_{2}\right) ^{2}-a_{2}^{2}-b_{2}^{2}}=\left( -1\right)^{a_{2}b_{2}}
\end{eqnarray*}
\item Part 4 is proved similarly to  part 3, using $\gamma =\gamma^{2^{n-1}} _1$.
\end{enumerate}

\end{proof}

Since $H$ is noncommutative, the action $\rightharpoonup $ is not trivial, and $t$ acts on $G$ as a group automorphism of order $2$. Let
\begin{eqnarray*}
t\rightharpoonup x &=& x^a y^b\\
t\rightharpoonup y &=& x^c y^{d}
\end{eqnarray*}
This action corresponds to  $A=\left[
\begin{array}{ll}
a & c \\
b & d%
\end{array}%
\right] \in GL_2 \left(\mathbb{Z}_{2^n}\right)$ such that $A\neq I$ and $A^2=I$.  Two actions are similar if and only if the corresponding matrices are similar.

\begin{theorem}\label{action}
For $n\geq 3$, there are $16$ non-similar actions of $t$  on $G$, corresponding to the matrices in (\ref{0101})--(\ref{100-1+2n-1}). Only 6 of these actions are non-similar when $n=2$.
\end{theorem}
\begin{proof}
Consider an action which corresponds to  $A=\left[
\begin{array}{ll}
a & c \\
b & d%
\end{array}%
\right] \in GL_2 \left(\mathbb{Z}_{2^n}\right)$. Since $A^2=I$,
$$
\begin{array}{ll}
a^2 + bc =1 & (a+d)c=0 \\
d^2 + bc =1 & (a+d)b=0%
\end{array}
$$
Therefore $a$ and $d$ have the same parity and $(a+d)(a-d)=0$.
\begin{enumerate}
\item Assume that $a$ and $d$ are both even. Then $b$ and $c$ are both odd and $d=-a$. The matrix $\left[
\begin{array}{ll}
1 & 0 \\
a & c
\end{array}%
\right]$ is invertible and
$$
\left[
\begin{array}{ll}
1 & 0 \\
a & c
\end{array}%
\right]
\left[
\begin{array}{ll}
a& c \\
b & -a
\end{array}%
\right] = \left[
\begin{array}{ll}
0& 1 \\
1 & 0
\end{array}%
\right]\left[
\begin{array}{ll}
1 & 0 \\
a & c
\end{array}%
\right]
$$
Thus all matrices of this type lie in one conjugacy class with representative
\begin{equation}
\left[
\begin{array}{ll}
0& 1 \\
1 & 0
\end{array}%
\right] \label{0101}
\end{equation}
\item Assume that $a$ and $d$ are both odd and $a+d=2q$ in $\mathbb{Z}_{2^{n}}$, where $q$ is odd.

Then $b,c \in \{ 0, 2^{n-1}\}$ and $a^2=d^2 =1$. Therefore $a, d \in\{\pm 1, \pm 1 + 2^{n-1}\}$. Thus either $d=a$ or $d= a+2^{n-1}$ (the latter case is possible only for $n\geq 3$, since when $n=2$, $a+(a+2)=0$).
    Since
   \begin{eqnarray*}
    \left[
\begin{array}{ll}
0& 1 \\
1 & 0
\end{array}%
\right]
\left[
\begin{array}{ll}
a& c \\
b & d
\end{array}%
\right] &=& \left[
\begin{array}{ll}
d& b \\
c & a
\end{array}%
\right]\left[
\begin{array}{ll}
0& 1 \\
1 & 0
\end{array}%
\right]\\
\left[
\begin{array}{ll}
1& 1 \\
0 & 1
\end{array}%
\right]
\left[
\begin{array}{ll}
a& 2^{n-1} \\
2^{n-1} & a
\end{array}%
\right] &=& \left[
\begin{array}{ll}
a+2^{n-1}& 0 \\
2^{n-1} & a+2^{n-1}
\end{array}%
\right]\left[
\begin{array}{ll}
1& 1 \\
0 & 1
\end{array}%
\right]\\
\left[
\begin{array}{ll}
1& 1 \\
1 & 0
\end{array}%
\right]
\left[
\begin{array}{ll}
a& 0 \\
2^{n-1} & a + 2^{n-1}
\end{array}%
\right] &=& \left[
\begin{array}{ll}
a+2^{n-1}& 0 \\
0 & a
\end{array}%
\right]\left[
\begin{array}{ll}
1& 1 \\
1 & 0
\end{array}%
\right]
\end{eqnarray*}
every matrix of this type is similar to one of the following 11 matrices:
\begin{eqnarray}
\left[
\begin{array}{ll}
-1 & 0 \\
0 & -1%
\end{array}%
\right]; \quad \left[
\begin{array}{ll}
1+2^{n-1} & 0 \\
0 & 1+2^{n-1}%
\end{array}%
\right]; \quad
 \left[
\begin{array}{cc}
-1+2^{n-1} & 0 \\
0 & -1+2^{n-1}%
\end{array}%
\right] ; \label{-10-10}\\
\left[
\begin{array}{cc}
1 & 0\\
 2^{n-1} & 1%
\end{array}%
\right] ; \quad  \left[
\begin{array}{cc}
-1 & 0\\
 2^{n-1} & -1%
\end{array}%
\right] ;\quad  \left[
\begin{array}{cc}
1+2^{n-1} & 0\\
 2^{n-1} & 1+2^{n-1}%
\end{array}%
\right] ;\quad  \left[
\begin{array}{cc}
-1+2^{n-1} & 0\\
 2^{n-1} & -1+2^{n-1}%
\end{array}%
\right] ;\notag\\
\left[
\begin{array}{cc}
1 & 0 \\
0 & 1+2^{n-1}%
\end{array}%
\right] ; \quad \left[
\begin{array}{cc}
-1 & 0 \\
0 & -1+2^{n-1}%
\end{array}%
\right] ; \quad \left[
\begin{array}{cc}
1 & 2^{n-1} \\
2^{n-1} & 1+2^{n-1}%
\end{array}%
\right] ; \quad
\left[
\begin{array}{cc}
-1 & 2^{n-1} \\
2^{n-1} & -1+2^{n-1}%
\end{array}%
\right] \notag
\end{eqnarray}
\item Assume that $a$ and $d$ are both odd and  $a+d=4q$ in $\mathbb{Z}_{2^{n}}$.

If $n=2$ the assumption means that $a=-d$.

If $n\geq 3$,  $a-d=2(a-2q)$ where $a-2q$ is odd. Since $(a+d)(a-d)=0$ in $\mathbb{Z}_{2^{n}}$, $a+d \in \{ 0, 2^{n-1}\}$. Thus either $d=-a$ or $d= -a+2^{n-1}$.  Therefore $a, d \in\{\pm 1, \pm 1 + 2^{n-1}\}$.
\begin{enumerate}
\item Assume that $b$ or $c$ is odd. Then $d=-a$ (since $(a+d)b=(a+d)c=0$). Without loss of generality, $b$ is odd. Then the matrix $\left[
\begin{array}{ll}
0 & 1 \\
b & -a
\end{array}%
\right]$ is invertible and
$$
\left[
\begin{array}{ll}
0 & 1 \\
b & -a
\end{array}%
\right]
\left[
\begin{array}{ll}
a& c \\
b & -a
\end{array}%
\right] = \left[
\begin{array}{ll}
0& 1 \\
1 & 0
\end{array}%
\right]\left[
\begin{array}{ll}
0 & 1 \\
b & -a
\end{array}%
\right]
$$
Thus every matrix of this type is similar to $\left[
\begin{array}{ll}
0& 1 \\
1 & 0
\end{array}%
\right]$.
\item Assume that $b$ and $c$ are even and $d=-a$.

If $n=2$, every matrix of this type is similar to $\left[
\begin{array}{ll}
1 & 0 \\
0 & -1
\end{array}%
\right]$ or $\left[
\begin{array}{ll}
1 & 2 \\
2 & -1
\end{array}%
\right]$.
Assume now that $n\geq 3$. Since $\left[
\begin{array}{ll}
a& c \\
b & -a
\end{array}%
\right]$ is similar to $\left[
\begin{array}{ll}
-a& b \\
c & a
\end{array}%
\right]$ we may assume that $a\equiv 1 \pmod 4$. Then the matrices
$$
\left[
\begin{array}{ll}
\frac{a+1}{2}& \frac{c}{2} \\
-\frac{b}{2} & \frac{a+1}{2}
\end{array}%
\right] \quad \text{and} \quad
\left[
\begin{array}{ll}
\frac{a+1}{2}+2^{n-2} & \frac{c}{2} \\
-\frac{b}{2} & \frac{a+1}{2}+2^{n-2}
\end{array}%
\right]
$$
are invertible. Since $a^2+bc=1$, $\frac{a^2+bc-1}{2} =0$ or $2^{n-1}$.

When $\frac{a^2+bc-1}{2} =0$
$$
   \left[
\begin{array}{ll}
\frac{a+1}{2}& \frac{c}{2} \\
-\frac{b}{2} & \frac{a+1}{2}
\end{array}%
\right]
\left[
\begin{array}{ll}
a& c \\
b & -a
\end{array}%
\right] = \left[
\begin{array}{ll}
1& 0 \\
0 & -1
\end{array}%
\right]\left[
\begin{array}{ll}
\frac{a+1}{2}& \frac{c}{2} \\
-\frac{b}{2} & \frac{a+1}{2}
\end{array}%
\right]
$$

When $\frac{a^2+bc-1}{2} =2^{n-1}$
$$
 \left[
\begin{array}{ll}
\frac{a+1+2^{n-1}}{2} & \frac{c}{2} \\
-\frac{b}{2} & \frac{a+1+2^{n-1}}{2}
\end{array}%
\right]
\left[
\begin{array}{ll}
a& c \\
b & -a
\end{array}%
\right] = \left[
\begin{array}{ll}
1+2^{n-1}& 0 \\
0 & -1+2^{n-1}
\end{array}%
\right]\left[
\begin{array}{ll}
\frac{a+1+2^{n-1}}{2} & \frac{c}{2} \\
-\frac{b}{2} & \frac{a+1+2^{n-1}}{2}
\end{array}%
\right]
$$
Therefore every matrix of this type is similar to
\begin{equation}
\left[
\begin{array}{ll}
1& 0 \\
0 & -1
\end{array}%
\right]\quad \text{or} \quad
\left[
\begin{array}{ll}
1+2^{n-1}& 0 \\
0 & -1+2^{n-1}
\end{array}%
\right] \label{100-1}
\end{equation}
\item Assume that $b$ and $c$ are even and $d=-a+2^{n-1}$.
Then $n\geq 3$ and as in the previous case we may assume that $a\equiv 1 \pmod 4$. Then the matrices
$$
\left[
\begin{array}{ll}
\frac{a+1}{2}& \frac{c}{2} \\
-\frac{b}{2} & \frac{a+1}{2}
\end{array}%
\right] \quad \text{and} \quad
\left[
\begin{array}{ll}
\frac{a+1+2^{n-1}}{2} & \frac{c}{2} \\
-\frac{b}{2} & \frac{a+1+2^{n-1}}{2}
\end{array}%
\right]
$$
are invertible. Since $a^2+bc=1$, $\frac{a^2+bc-1}{2} =0$ or $2^{n-1}$.

When $\frac{a^2+bc-1}{2} =0$
$$
 \left[
\begin{array}{ll}
\frac{a+1+2^{n-1}}{2} & \frac{c}{2} \\
-\frac{b}{2} & \frac{a+1+2^{n-1}}{2}
\end{array}%
\right]
\left[
\begin{array}{ll}
a& c \\
b & -a
\end{array}%
\right] = \left[
\begin{array}{ll}
1& 0 \\
0 & -1+2^{n-1}
\end{array}%
\right]\left[
\begin{array}{ll}
\frac{a+1+2^{n-1}}{2} & \frac{c}{2} \\
-\frac{b}{2} & \frac{a+1+2^{n-1}}{2}
\end{array}%
\right]
$$

When $\frac{a^2+bc-1}{2} =2^{n-1}$
$$
   \left[
\begin{array}{ll}
\frac{a+1}{2}& \frac{c}{2} \\
-\frac{b}{2} & \frac{a+1}{2}
\end{array}%
\right]
\left[
\begin{array}{ll}
a& c \\
b & -a
\end{array}%
\right] = \left[
\begin{array}{ll}
1+2^{n-1}& 0 \\
0 & -1
\end{array}%
\right]\left[
\begin{array}{ll}
\frac{a+1}{2}& \frac{c}{2} \\
-\frac{b}{2} & \frac{a+1}{2}
\end{array}%
\right]
$$

Therefore every matrix of this type is similar to
\begin{equation}
\left[
\begin{array}{ll}
1& 0 \\
0 & -1+2^{n-1}
\end{array}%
\right]\quad \text{or} \quad
\left[
\begin{array}{ll}
1+2^{n-1}& 0 \\
0 & -1
\end{array}%
\right] \label{100-1+2n-1}
\end{equation}
\end{enumerate}
\end{enumerate}
\end{proof}

Hopf algebras associated with two different actions are isomorphic only if the corresponding matrices are similar. Thus we will consider the following  $16$ cases, corresponding to the matrices in (\ref{0101})--(\ref{100-1+2n-1}) (only 6 of these cases are possible when $n=2$):
\begin{theorem}\label{cohomology}
Assume that  $H$ is  a semisimple Hopf algebra of dimension $2^{2n+1}$ with group of group-like elements isomorphic to $%
\mathbb{Z}_{2^{n}}\times \mathbb{Z}_{2^{n}}$ for $n\geq 2$, which corresponds to the action associated with one of the matrices in (\ref{0101})--(\ref{100-1+2n-1}). Then, for $\Psi$ defined in Theorem \ref{psi}, we have
\begin{eqnarray*}
{\rm H}^{2}\left( kL,k^{G}\right) &\cong &
 {\rm H}_{c}^{2}\left( kL,k^{G}\right) \cong  {\rm Ker}  \Psi \times {\rm Im} \Psi \\
 &\cong& \hat G ^L/ N_t\left(\hat G\right)\times M(G)^-\cong {\rm H}^{2}\left( L,\hat G\right)\times M(G)^{-}
\end{eqnarray*}
 \end{theorem}

In order to prove this theorem we will consider the following sixteen cases corresponding to the matrices in (\ref{0101})--(\ref{100-1+2n-1}) (only six of them are possible when $n=2$) and establish the results in each of the cases separately.

\subsection{Case 1: $t\rightharpoonup x = y, \quad
t\rightharpoonup y = x$}\quad

\begin{proposition}\label{xy} Let $\omega$ be a primitive $2^{n+1}$-th root of $1$ such that $\omega^2 =\xi$.
 $$
 {\rm H}^{2}\left( kL,k^{G}\right)= \left\langle [(\sigma^{\beta} , \tau ^{j}_{1,2})]\right\rangle \cong {\rm Im} \Psi \cong \left\langle [\alpha^{\prime} _{1,2}] \right\rangle  \cong \mathbb{Z}_{2^{n}} .
 $$
 where
 $$
 \alpha _{1,2}^{\prime }\left( x^{a_{1}}y^{a_{2}},x^{b_{1}}y^{b_{2}}\right)
 =\omega^{a_{1}b_{2}-a_{2}b_{1}}(-1)^{(a_1+b_1)k_2+(a_2+b_2)k_1}
 $$
with $k_j = 0$ if $0\leq a_j+b_j < 2^{n}$ and $k_j = 1$ if $2^{n}\leq a_j+b_j < 2^{n+1}$ and
 $$
 \beta\left( x^{a_{1}}y^{a_{2}}\right) =\xi^{a_1 a_2}.
 $$
\end{proposition}
\begin{proof}
Since $t\rightharpoonup x = y $ and $
t\rightharpoonup y = x$,

\begin{eqnarray*}
t\rightharpoonup \alpha _{1,2} &=& \alpha _{2,1} \\
t\rightharpoonup \beta_{1}&=& \beta_{2}\\
t\rightharpoonup \beta_{2}&=& \beta_{1}
\end{eqnarray*}
By part 1 of Proposition \ref{12_21}, $\left[ \alpha _{2,1}\right] =\left[ \alpha ^{-1}_{1,2}\right]$ in $M\left( G\right)$. Therefore $M\left( G\right)^{-} =M\left( G\right)$. We will now define $\alpha _{1,2}^{\prime } \in {\rm Z}^2\left( G, k^{\times}\right) $ such that $\left[ \alpha _{1,2}\right] = \left[ \alpha _{1,2}^{\prime }\right] $ and $\alpha^{\prime } _{1,2}\left( t\rightharpoonup\alpha^{\prime } _{1,2}\right)=1$.

Let $\omega$ be a primitive $2^{n+1}$-th root of $1$ such that $\omega^2 =\xi$. Define $\gamma
\in \left( k^{G}\right) ^{\times }$ and $\alpha _{1,2}^{\prime } \in {\rm Z}^2\left( G, k^{\times}\right) $ via
\begin{eqnarray}
\gamma \left( x^{a_{1}}y^{a_{2}}\right)
&=&\omega ^{a_{1}a_{2}}  \label{gamma}\\
\alpha _{1,2}^{\prime }\left( g,h \right)&=&\alpha _{1,2}\left( g,h \right)\gamma \left( g\right) \gamma \left( h\right) \gamma \left( gh\right)
^{-1}\notag
\end{eqnarray}
Then $\left[ \alpha _{1,2}^{\prime }\right] = \left[ \alpha _{1,2}\right]$ and
\begin{eqnarray*}
\alpha _{1,2}^{\prime }\left( x^{a_{1}}y^{a_{2}},x^{b_{1}}y^{b_{2}}\right) &=& \xi^{a_{1}b_{2}}\omega ^{a_{1}a_{2}}\omega ^{b_{1}b_{2}} \omega^{-(a_{1}+b_{1} -2^{n}k_1)(a_{2}+b_{2} -2^{n}k_2)}\\
&=& \omega^{a_{1}b_{2}-a_{2}b_{1}}(-1)^{(a_1+b_1)k_2+(a_2+b_2)k_1}
\end{eqnarray*}
where $k_j = 0$ if $0\leq a_j+b_j < 2^{n}$ and $k_j = 1$ if $2^{n}\leq a_j+b_j < 2^{n+1}$.
\begin{eqnarray*}
&&\alpha^{\prime } _{1,2}\left( x^{a_{1}}y^{a_{2}},x^{b_{1}}y^{b_{2}}\right)\left( t\rightharpoonup\alpha^{\prime } _{1,2}\right)\left( x^{a_{1}}y^{a_{2}},x^{b_{1}}y^{b_{2}}\right)\\
&& =\xi^{a_{1}b_{2}}\omega ^{a_{1}a_{2}}\omega ^{b_{1}b_{2}} \omega^{-(a_{1}+b_{1} -2^{n}k_1)(a_{2}+b_{2} -2^{n}k_2)}
 \xi ^{a_{2}b_{1}}\omega ^{a_{2}a_{1}}\omega ^{b_{2}b_{1}} \omega^{-(a_{2}+b_{2} -2^{n}k_2)(a_{1}+b_{1} -2^{n}k_1)} \\
&&= \xi^{a_{1}b_{2}+a_{2}b_{1}}\xi ^{a_{1}a_{2}}\xi ^{b_{1}b_{2}} \xi^{-(a_{1}+b_{1} -2^{n}k_1)(a_{2}+b_{2} -2^{n}k_2)}=\xi^{a_{1}b_{2}+a_{2}b_{1}}\xi ^{a_{1}a_{2}}\xi ^{b_{1}b_{2}} \xi^{-a_{1}a_{2}-a_{1}b_{2} -a_{2}b_{1}-b_{1}b_{2}}=1
\end{eqnarray*}
Thus $\alpha^{\prime } _{1,2}\left( t\rightharpoonup\alpha^{\prime } _{1,2}\right)=1$ and therefore $\left[ \alpha _{1,2}^{\prime }\right] = \left[ \alpha _{1,2}\right] \in {\rm Im} \Psi$.

 Then
\begin{itemize}
\item ${\rm Im} \Psi = M\left( G\right) = \left\langle [\alpha ^{\prime }_{1,2}] \right\rangle$.
\item $\hat G ^L = \left\langle \beta_{1}\beta_{2} \right\rangle $ and $ N_t\left(\hat G\right) = \left\langle \beta_{1}\beta_{2}\right\rangle $.
\item ${\rm Ker}  \Psi = \{ 1\}$.

 \end{itemize}

Define $(\tau^{\prime } _{p,q})^j$ via $\left((\tau^{\prime } _{p,q})^j\right) _t = (\alpha^{\prime } _{1,2})^j$. Then $(\tau^{\prime } _{p,q})^j$ is a dual cocycle. Define $\beta  \in \left( k^{G}\right) ^{\times }$ via
$$
\beta\left( x^{a_{1}}y^{a_{2}}\right) =\xi^{a_1 a_2}
$$
Since $\beta (g) =\gamma \left( g\right) \gamma
\left( t\rightharpoonup g\right)$ for $\gamma \in \left( k^{G}\right) ^{\times }$ defined in formula (\ref{gamma}),
$$
 [( {\rm triv},(\tau^{\prime } _{1,2})^j)]=[(\sigma^{\beta ^j} , \tau ^{j}_{1,2})]   \in {\rm H}^{2}\left( kL,k^{G}\right)
 $$

 Thus
 $$
 {\rm H}^{2}\left( kL,k^{G}\right) = \left\langle [( {\rm triv},\tau^{\prime } _{1,2})]\right\rangle = \left\langle [(\sigma^{\beta} , \tau ^{j}_{1,2})]\right\rangle \cong \left\langle [\alpha^{\prime} _{1,2}] \right\rangle  \cong \mathbb{Z}_{2^{n}}
$$
 \end{proof}


\subsection{Case 2: $t\rightharpoonup x = x^{-1}, \quad t\rightharpoonup y = y^{-1}$}\quad


 \begin{proposition} \label{-1-1} In this case
 $$
 {\rm H}^{2}\left( kL,k^{G}\right)\cong {\rm Ker}  \Psi \times {\rm Im} \Psi \cong \left\langle [\alpha^{2^{n-1}} _{1,1}] \right\rangle \times\left\langle [\alpha^{2^{n-1}} _{2,2}] \right\rangle \times\left\langle [\alpha^{2^{n-1}} _{1,2}] \right\rangle \cong \mathbb{Z}_{2} \times\mathbb{Z}_{2} \times \mathbb{Z}_{2}.
 $$
\end{proposition}
\begin{proof}
Since $t\rightharpoonup x = x^{-1}$ and $ t\rightharpoonup y = y^{-1}$,
\begin{eqnarray*}
t\rightharpoonup \alpha _{1,2} &=& \alpha_{1,2}\\
t\rightharpoonup \beta_{j}&=& \beta_{j}^{-1} \quad \text{ for } j=1,2
\end{eqnarray*}
Thus
\begin{itemize}
\item ${\rm Im} \Psi = M\left( G\right)^{-} = \left\langle [\alpha^{2^{n-1}} _{1,2}] \right\rangle$.
\item $\hat G ^L = \left\langle \beta^{2^{n-1}} _1 \right\rangle \times \left\langle \beta^{2^{n-1}} _2 \right\rangle $ and $ N_t\left(\hat G\right) = \{1\} $.
\item ${\rm Ker}  \Psi = \left\langle [(\sigma^{\beta^{2^{n-1}} _1} , {\rm triv})] \right\rangle \times \left\langle [(\sigma^{\beta^{2^{n-1}} _2} , {\rm triv})] \right\rangle \cong \mathbb{Z}_{2} \times \mathbb{Z}_{2}$.
\end{itemize}
 By parts 3 and 4 of Proposition \ref{12_21}, $[(\sigma^{\beta^{2^{n-1}} _j} , {\rm triv})] =  [( {\rm triv}, \tau ^{2^{n-1}}_{j,j} )] $ for $j=1,2$.
Thus
$$
{\rm Ker}  \Psi = \left\langle [( {\rm triv}, \tau ^{2^{n-1}}_{1,1} )]  \right\rangle \times \left\langle [( {\rm triv}, \tau ^{2^{n-1}}_{2,2} )] \right\rangle\cong
\left\langle [\alpha^{2^{n-1}} _{1,1}] \right\rangle \times\left\langle [\alpha^{2^{n-1}} _{2,2}] \right\rangle
$$
and
$$
 {\rm H}^{2}\left( kL,k^{G}\right) \cong
 {\rm H}_{c}^{2}\left( kL,k^{G}\right) \cong \left\langle [\alpha^{2^{n-1}} _{1,1}] \right\rangle \times\left\langle [\alpha^{2^{n-1}} _{1,2}] \right\rangle \times\left\langle [\alpha^{2^{n-1}} _{2,2}] \right\rangle \cong \mathbb{Z}_{2} \times\mathbb{Z}_{2} \times \mathbb{Z}_{2}.
 $$
\end{proof}
\subsection{Case 3: $t\rightharpoonup x = x^{1+2^{n-1}}, \quad t\rightharpoonup y = y^{1+2^{n-1}}$} \quad

In this case we assume that $n\geq 3$, since for $n=2$ the matrix corresponding to the action of $t$ is equal to the matrix from Case 2. Then
 \begin{proposition} In this case
 $$
 {\rm H}^{2}\left( kL,k^{G}\right)\cong {\rm Im} \Psi \cong \left\langle [\alpha^{2^{n-1}} _{1,2}] \right\rangle  \cong \mathbb{Z}_{2}.
 $$
\end{proposition}
\begin{proof}
Since $t\rightharpoonup x = x^{1+2^{n-1}}$ and $t\rightharpoonup y = y^{1+2^{n-1}}$,
\begin{eqnarray*}
t\rightharpoonup \alpha _{1,2} &=& \alpha_{1,2}\\
t\rightharpoonup \beta_{j}&=& \beta_{j}^{ 1+2^{n-1}}\quad \text{ for } j=1,2
\end{eqnarray*}

Thus
\begin{itemize}
\item ${\rm Im} \Psi = M\left( G\right)^{-} = \left\langle [\alpha^{2^{n-1}} _{1,2}] \right\rangle$.
\item $\hat G ^L = \left\langle \beta^{2} _1 \right\rangle \times \left\langle \beta^{2} _2 \right\rangle $ and $ N_t\left(\hat G\right) =\left\langle \beta^{2} _1 \right\rangle \times \left\langle \beta^{2} _2 \right\rangle$.
\item ${\rm Ker}  \Psi = \{1\} $.
\end{itemize}

Therefore
$$
 {\rm H}^{2}\left( kL,k^{G}\right) \cong
 {\rm H}_{c}^{2}\left( kL,k^{G}\right) \cong \left\langle [\alpha^{2^{n-1}} _{1,2}] \right\rangle  \cong \mathbb{Z}_{2}.
 $$
\end{proof}


\subsection{Case 4: $t\rightharpoonup x = x^{-1+2^{n-1}}, \quad t\rightharpoonup y = y^{-1+2^{n-1}}$}\quad

In this case we assume that $n\geq 3$, since for $n=2$ this  action is trivial.

 \begin{proposition} In this case
 $$
 {\rm H}^{2}\left( kL,k^{G}\right)\cong {\rm Im} \Psi \cong \left\langle [\alpha^{2^{n-1}} _{1,2}] \right\rangle  \cong \mathbb{Z}_{2}.
 $$
\end{proposition}
\begin{proof}
Since $t\rightharpoonup x = x^{-1+2^{n-1}}$ and $ t\rightharpoonup y = y^{-1+2^{n-1}}$,
\begin{eqnarray*}
t\rightharpoonup \alpha _{1,2} &=& \alpha_{1,2}\\
t\rightharpoonup \beta_{j}&=& \beta_{j}^{ -1+2^{n-1}}\quad \text{ for } j=1,2
\end{eqnarray*}
Thus
\begin{itemize}
\item ${\rm Im} \Psi = M\left( G\right)^{-} = \left\langle [\alpha^{2^{n-1}} _{1,2}] \right\rangle$.
\item $\hat G ^L = \left\langle \beta^{2^{n-1}} _1 \right\rangle \times \left\langle \beta^{2^{n-1}} _2 \right\rangle $ and $ N_t\left(\hat G\right) =\left\langle \beta^{2^{n-1}} _1 \right\rangle \times \left\langle \beta^{2^{n-1}} _2 \right\rangle  $.
\item ${\rm Ker}  \Psi = \{1\} $.
\end{itemize}

Therefore
$$
 {\rm H}^{2}\left( kL,k^{G}\right) \cong
 {\rm H}_{c}^{2}\left( kL,k^{G}\right) \cong \left\langle [\alpha^{2^{n-1}} _{1,2}] \right\rangle  \cong \mathbb{Z}_{2}.
 $$
\end{proof}


\subsection{Case 5: $t\rightharpoonup x = x, \quad t\rightharpoonup y = y^{1+2^{n-1}}$}\quad

In this case we assume that $n\geq 3$, since for $n=2$ the matrix corresponding to the action of $t$ is equal to the matrix from Case 7.

\begin{proposition} In this case
 $$
{\rm H}^{2}\left( kL,k^{G}\right) = \left\langle [(\sigma^{\beta _1} , {\rm triv})] \right\rangle   \times\left\langle [({\rm triv},\tau^{2^{n-1}} _{1,2})] \right\rangle \cong {\rm Ker}  \Psi \times {\rm Im} \Psi\cong \left\langle [\alpha '] \right\rangle\times \left\langle [\alpha^{2^{n-1}} _{1,2}] \right\rangle\cong \mathbb{Z}_{2} \times\mathbb{Z}_{2}
$$ for $\alpha '$ defined via
$$
\alpha ' \left( x^{a_{1}}y^{a_{2}},x^{b_{1}}y^{b_{2}}\right) =\left( -1\right)
^{k}
$$
where $k= 0$ if $0\leq a_1+b_1 < 2^{n}$ and $k = 1$ if $2^{n}\leq a_1+b_1 < 2^{n+1}$.
\end{proposition}
\begin{proof}
Since $t\rightharpoonup x = x$ and $ t\rightharpoonup y = y^{1+2^{n-1}}$,
\begin{eqnarray*}
t\rightharpoonup \alpha _{1,2} &=& \alpha_{1,2}^{ 1+2^{n-1}}\\
t\rightharpoonup \beta_{1}&=& \beta_{1}\\
t\rightharpoonup \beta_{2}&=& \beta_{2}^{ 1+2^{n-1}}
\end{eqnarray*}

Thus
\begin{itemize}
\item ${\rm Im} \Psi = M\left( G\right)^{-} =  \left\langle [\alpha^{2^{n-1}} _{1,2}] \right\rangle$.
\item $\hat G ^L = \left\langle \beta _1 \right\rangle \times \left\langle \beta^{2} _2 \right\rangle $ and $ N_t\left(\hat G\right) =\left\langle \beta^{2} _1 \right\rangle \times \left\langle \beta^{2} _2 \right\rangle $.
\item ${\rm Ker}  \Psi = \left\langle [(\sigma^{\beta _1} , {\rm triv})] \right\rangle \cong \mathbb{Z}_{2}$.
\end{itemize}

 By part 2 of Proposition \ref{12_21},  $[(\sigma^{\beta_1} , {\rm triv})] =  [( {\rm triv}, \tau '  )] $ where $\tau '_t =\alpha '$. Therefore
$$
{\rm H}^{2}\left( kL,k^{G}\right) = \left\langle [(\sigma^{\beta _1} , {\rm triv})] \right\rangle   \times\left\langle [({\rm triv},\tau^{2^{n-1}} _{1,2})] \right\rangle \cong \left\langle [\alpha '] \right\rangle\times \left\langle [\alpha^{2^{n-1}} _{1,2}] \right\rangle\cong \mathbb{Z}_{2} \times\mathbb{Z}_{2}.
$$
\end{proof}

\subsection{Case 6: $t\rightharpoonup x = x^{-1}, \quad t\rightharpoonup y = y^{-1+2^{n-1}}$}\quad

In this case we assume that $n\geq 3$, since for $n=2$ the matrix corresponding to the action of $t$ is similar to the matrix from Case 7.

 \begin{proposition} In this case
 $$
 {\rm H}^{2}\left( kL,k^{G}\right)\cong {\rm Ker}  \Psi \times {\rm Im} \Psi \cong \left\langle \alpha^{2^{n-1}} _{1,1} \right\rangle \times\left\langle \alpha^{2^{n-1}} _{1,2} \right\rangle  \cong \mathbb{Z}_{2} \times\mathbb{Z}_{2}.
 $$
\end{proposition}
\begin{proof}
Since $t\rightharpoonup x = x^{-1}$ and $ t\rightharpoonup y = y^{-1+2^{n-1}}$,
\begin{eqnarray*}
t\rightharpoonup \alpha _{1,2} &=& \alpha_{1,2}^{ 1+2^{n-1}}\\
t\rightharpoonup \beta_{1}&=& \beta_{1}^{ -1}\\
t\rightharpoonup \beta_{2}&=& \beta_{2}^{ -1+2^{n-1}}
\end{eqnarray*}
Thus
\begin{itemize}
\item ${\rm Im} \Psi = M\left( G\right)^{-} = \left\langle [\alpha^{2^{n-1}} _{1,2}] \right\rangle $.
\item $\hat G ^L = \left\langle \beta^{2^{n-1}} _1 \right\rangle \times \left\langle \beta^{2^{n-1}} _2 \right\rangle $ and $ N_t\left(\hat G\right) = \left\langle \beta^{2^{n-1}} _2 \right\rangle $.
\item ${\rm Ker}  \Psi = \left\langle [(\sigma^{\beta^{2^{n-1}} _1} , {\rm triv})] \right\rangle \cong \mathbb{Z}_{2}$.
 \end{itemize}
 By part 4 of Proposition \ref{12_21},   $[(\sigma^{\beta^{2^{n-1}} _1} , {\rm triv})] =  [( {\rm triv}, \tau ^{2^{n-1}}_{1,1} )] $. Thus
$$
{\rm Ker}  \Psi = \left\langle [( {\rm triv}, \tau ^{2^{n-1}}_{1,1} )]  \right\rangle  \cong
\left\langle \alpha^{2^{n-1}} _{1,1} \right\rangle
$$
and
$$
 {\rm H}^{2}\left( kL,k^{G}\right) \cong
 {\rm H}_{c}^{2}\left( kL,k^{G}\right) \cong \left\langle \alpha^{2^{n-1}} _{1,1} \right\rangle \times\left\langle \alpha^{2^{n-1}} _{1,2} \right\rangle  \cong \mathbb{Z}_{2} \times\mathbb{Z}_{2} .
 $$
\end{proof}


\subsection{Case 7: $t\rightharpoonup x = x \quad
t\rightharpoonup y = y^{-1}$}\quad

\begin{proposition}\label{1-1} In this case
\begin{eqnarray*}
 {\rm H}^{2}\left( kL,k^{G}\right) &=& \left\langle [(\sigma^{\beta _1} , {\rm triv})] \right\rangle \times \left\langle [(\sigma^{\beta^{2^{n-1}} _2} , {\rm triv})]\right\rangle  \times\left\langle [({\rm triv}, \tau _{1,2})] \right\rangle \cong {\rm Ker}  \Psi \times {\rm Im} \Psi\\
 &\cong & \left\langle \alpha ' \right\rangle\times \left\langle \alpha^{2^{n-1}} _{2,2} \right\rangle \times \left\langle \alpha _{1,2} \right\rangle \cong \mathbb{Z}_{2} \times\mathbb{Z}_{2} \times \mathbb{Z}_{2^{n}}
 \end{eqnarray*}
  for $\alpha '$ defined via
$$
\alpha ' \left( x^{a_{1}}y^{a_{2}},x^{b_{1}}y^{b_{2}}\right) =\left( -1\right)
^{k}
$$
where $k= 0$ if $0\leq a_1+b_1 < 2^{n}$ and $k = 1$ if $2^{n}\leq a_1+b_1 < 2^{n+1}$.
\end{proposition}
\begin{proof}
Since $t\rightharpoonup x = x $ and $
t\rightharpoonup y = y^{-1}$,
\begin{eqnarray*}
t\rightharpoonup \alpha _{1,2} &=& \alpha^{-1}_{1,2}\\
t\rightharpoonup \beta_{1}&=& \beta_{1}\\
t\rightharpoonup \beta_{2}&=& \beta_{2}^{-1}
\end{eqnarray*}

Thus
\begin{itemize}
\item ${\rm Im} \Psi = M\left( G\right)$.
\item $\hat G ^L = \left\langle \beta _1 \right\rangle \times \left\langle \beta^{2^{n-1}} _2 \right\rangle $ and $ N_t\left(\hat G\right) = \left\langle \beta^2 _1 \right\rangle $.
\item ${\rm Ker}  \Psi = \left\langle [(\sigma^{\beta _1} , {\rm triv})] \right\rangle \times \left\langle [(\sigma^{\beta^{2^{n-1}} _2} , {\rm triv})] \right\rangle \cong \mathbb{Z}_{2} \times \mathbb{Z}_{2}$.
\item ${\rm H}^{2}\left( kL,k^{G}\right) = \left\langle [(\sigma^{\beta _1} , {\rm triv})] \right\rangle \times \left\langle [(\sigma^{\beta^{2^{n-1}} _2} , {\rm triv})]\right\rangle  \times\left\langle [({\rm triv}, \tau _{1,2})] \right\rangle$.
 \end{itemize}

By parts 2 and 4 of Proposition \ref{12_21}, $[(\sigma^{\beta^{2^{n-1}} _2} , {\rm triv})] =  [( {\rm triv}, \tau ^{2^{n-1}}_{2,2} )] $ and $[(\sigma^{\beta_1} , {\rm triv})] =  [( {\rm triv}, \tau '  )] $ where $\tau '_t =\alpha '$. Therefore
$$
 {\rm H}^{2}\left( kL,k^{G}\right) \cong
 {\rm H}_{c}^{2}\left( kL,k^{G}\right)\cong\left\langle \alpha ' \right\rangle\times \left\langle \alpha^{2^{n-1}} _{2,2} \right\rangle \times \left\langle \alpha _{1,2} \right\rangle \cong \mathbb{Z}_{2} \times\mathbb{Z}_{2} \times \mathbb{Z}_{2^{n}}.
 $$

\end{proof}

\subsection{Case 8: $t\rightharpoonup x = x^{1+2^{n-1}} \quad
t\rightharpoonup y = y^{-1+2^{n-1}}$}\quad

In this case we assume that $n\geq 3$, since for $n=2$ the matrix corresponding to the action of $t$ is similar to the matrix from Case 7.

\begin{proposition} In this case
 $$
 {\rm H}^{2}\left( kL,k^{G}\right) \cong  {\rm Im} \Psi\cong \left\langle \alpha _{1,2} \right\rangle  \cong \mathbb{Z}^{2^{n}}_{2}.
 $$
\end{proposition}
\begin{proof}
Since $t\rightharpoonup x = x^{1+2^{n-1}} $ and $
t\rightharpoonup y = y^{-1+2^{n-1}}$,
\begin{eqnarray*}
t\rightharpoonup \alpha _{1,2} &=& \alpha^{-1}_{1,2}\\
t\rightharpoonup \beta_{1}&=& \beta^{1+2^{n-1}}_{1}\\
t\rightharpoonup \beta_{2}&=& \beta_{2}^{-1+2^{n-1}}
\end{eqnarray*}
Thus
\begin{itemize}
\item ${\rm Im} \Psi = M\left( G\right)$.
\item $\hat G ^L = \left\langle \beta^2 _1 \right\rangle \times \left\langle \beta^{2^{n-1}} _2 \right\rangle $ and $ N_t\left(\hat G\right) = \left\langle \beta^2 _1 \right\rangle \times \left\langle \beta^{2^{n-1}} _2 \right\rangle $.
\item ${\rm Ker}  \Psi = \{1\}$.
 \end{itemize}
Therefore
$$
{\rm H}^{2}\left( kL,k^{G}\right) \cong
 {\rm H}_{c}^{2}\left( kL,k^{G}\right) \cong \left\langle \alpha _{1,2} \right\rangle  \cong  \mathbb{Z}^{2^{n}}_{2}.
 $$
 \end{proof}
\subsection{Case 9: $t\rightharpoonup x = x \quad
t\rightharpoonup y = y^{-1+2^{n-1}}$}\quad

In this case we assume that $n\geq 3$, since for $n=2$ this action is trivial.

\begin{proposition} In this case
  $$
{\rm H}^{2}\left( kL,k^{G}\right) = \left\langle [(\sigma^{\beta _1} , {\rm triv})] \right\rangle   \times\left\langle [({\rm triv},\tau^{2} _{1,2})] \right\rangle\cong {\rm Ker}  \Psi \times {\rm Im} \Psi \cong \left\langle \alpha ' \right\rangle\times \left\langle \alpha^{2} _{1,2} \right\rangle\cong \mathbb{Z}_{2} \times\mathbb{Z}_{2^{n-1}}
$$
for $\alpha '$ defined via
$$
\alpha ' \left( x^{a_{1}}y^{a_{2}},x^{b_{1}}y^{b_{2}}\right) =\left( -1\right)
^{k}
$$
where $k= 0$ if $0\leq a_1+b_1 < 2^{n}$ and $k = 1$ if $2^{n}\leq a_1+b_1 < 2^{n+1}$.
\end{proposition}
\begin{proof}
Since  $t\rightharpoonup x = x$ and
$t\rightharpoonup y = y^{-1+2^{n-1}}$,
\begin{eqnarray*}
t\rightharpoonup \alpha _{1,2} &=& \alpha^{-1+2^{n-1}}_{1,2}\\
t\rightharpoonup \beta_{1}&=& \beta^{1+2^{n-1}}_{1}\\
t\rightharpoonup \beta_{2}&=& \beta_{2}^{-1+2^{n-1}}
\end{eqnarray*}

Thus
\begin{itemize}
\item ${\rm Im} \Psi = M\left( G\right)^-=\left\langle [\alpha _{1,2}] ^{2}\right\rangle$.
\item $\hat G ^L = \left\langle \beta _1 \right\rangle \times \left\langle \beta^{2^{n-1}} _2 \right\rangle $ and $ N_t\left(\hat G\right) = \left\langle \beta^2 _1 \right\rangle \times \left\langle \beta^{2^{n-1}} _2 \right\rangle $.
\item ${\rm Ker}  \Psi = \left\langle [(\sigma^{\beta _1} , {\rm triv})] \right\rangle \cong \mathbb{Z}_{2} $.
 \end{itemize}

 By part 2 of Proposition \ref{12_21},   $[(\sigma^{\beta_1} , {\rm triv})] =  [( {\rm triv}, \tau '  )] $ where $\tau '_t =\alpha '$.
Therefore
 $$
{\rm H}^{2}\left( kL,k^{G}\right) = \left\langle [(\sigma^{\beta _1} , {\rm triv})] \right\rangle   \times\left\langle [({\rm triv},\tau^{2} _{1,2})] \right\rangle \cong \left\langle \alpha ' \right\rangle\times \left\langle \alpha^{2} _{1,2} \right\rangle\cong \mathbb{Z}_{2} \times\mathbb{Z}_{2^{n-1}}.
$$
 \end{proof}

\subsection{Case 10: $t\rightharpoonup x = x^{1+2^{n-1}} \quad
t\rightharpoonup y = y^{-1}$}\quad

In this case we assume that $n\geq 3$, since for $n=2$ the matrix corresponding to the action of $t$ is equal to the matrix from Case 2. Then

\begin{proposition} In this case
  $$
{\rm H}^{2}\left( kL,k^{G}\right) \cong {\rm Ker}  \Psi \times {\rm Im} \Psi\cong \left\langle \alpha^{2^{n-1}} _{2,2} \right\rangle \times\left\langle \alpha^{2} _{1,2} \right\rangle  \cong \mathbb{Z}_{2} \times\mathbb{Z}_{2}^{2^{n-1}} .
$$
\end{proposition}
\begin{proof}
Since  $t\rightharpoonup x = x^{1+2^{n-1}}$ and $
t\rightharpoonup y = y^{-1}$,
\begin{eqnarray*}
t\rightharpoonup \alpha _{1,2} &=& \alpha^{-1+2^{n-1}}_{1,2}\\
t\rightharpoonup \beta_{1}&=& \beta^{1+2^{n-1}}_{1}\\
t\rightharpoonup \beta_{2}&=& \beta_{2}^{-1}
\end{eqnarray*}
Thus
\begin{itemize}
\item ${\rm Im} \Psi = M\left( G\right)^-=\left\langle [\alpha _{1,2}] ^{2}\right\rangle$.
\item $\hat G ^L = \left\langle \beta^2 _1 \right\rangle \times \left\langle \beta^{2^{n-1}} _2 \right\rangle $ and $ N_t\left(\hat G\right) = \left\langle \beta^2 _1 \right\rangle $.
\item ${\rm Ker}  \Psi = \left\langle [(\sigma^{\beta^{2^{n-1}} _2} , {\rm triv})] \right\rangle \cong \mathbb{Z}_{2}$.
 \end{itemize}

By part 4 of Proposition \ref{12_21},   $[(\sigma^{\beta^{2^{n-1}} _2} , {\rm triv})] =  [( {\rm triv}, \tau ^{2^{n-1}}_{2,2} )] $. Thus
$$
{\rm Ker}  \Psi = \left\langle [( {\rm triv}, \tau ^{2^{n-1}}_{2,2} )]  \right\rangle  \cong
\left\langle \alpha^{2^{n-1}} _{2,2} \right\rangle
$$
and
$$
 {\rm H}^{2}\left( kL,k^{G}\right) \cong
 {\rm H}_{c}^{2}\left( kL,k^{G}\right) \cong \left\langle \alpha^{2^{n-1}} _{2,2} \right\rangle \times\left\langle \alpha^{2} _{1,2} \right\rangle  \cong \mathbb{Z}_{2} \times\mathbb{Z}_{2}^{2^{n-1}} .
 $$
\end{proof}

\subsection{Case 11: $t\rightharpoonup x = xy^{2^{n-1}}, \quad
t\rightharpoonup y = y$}\quad
\begin{proposition}\label{121} In this case
  $$
{\rm H}^{2}\left( kL,k^{G}\right) = \left\langle [(\sigma^{\beta _1} , {\rm triv})] \right\rangle   \times\left\langle [({\rm triv},\tau^{2} _{1,2})] \right\rangle\cong {\rm Ker}  \Psi \times {\rm Im} \Psi \cong \left\langle \alpha ' \right\rangle\times \left\langle \alpha^{2^{n-1}} _{1,2} \right\rangle\cong \mathbb{Z}_{2} \times\mathbb{Z}_{2^{n-1}}
$$
for $\alpha '$ defined via
$$
\alpha ' \left( x^{a_{1}}y^{a_{2}},x^{b_{1}}y^{b_{2}}\right) =\left( -1\right)
^{k}
$$
where $k= 0$ if $0\leq a_1+b_1 < 2^{n}$ and $k = 1$ if $2^{n}\leq a_1+b_1 < 2^{n+1}$.
\end{proposition}
\begin{proof}
Since $t\rightharpoonup x = xy^{2^{n-1}}$ and
$t\rightharpoonup y = y$,
\begin{eqnarray*}
t\rightharpoonup \alpha _{1,2} &=& \alpha _{1,1}^{2^{n-1}}\alpha _{1,2} \\
t\rightharpoonup \beta_{1}&=& \beta_{1}\\
t\rightharpoonup \beta_{2}&=& \beta_{1}^{2^{n-1}}\beta_{2} 
\end{eqnarray*}
Thus 
\begin{itemize}
\item ${\rm Im} \Psi = M\left( G\right)^{-} = \left\langle [\alpha^{2^{n-1}} _{1,2}] \right\rangle$.
\item $\hat G ^L = \left\langle \beta _1 \right\rangle \times \left\langle \beta^2 _2 \right\rangle $ and $ N_t\left(\hat G\right) = \left\langle \beta^2 _1 \right\rangle \times \left\langle \beta^2 _2 \right\rangle $.
\item ${\rm Ker}  \Psi = \left\langle [(\sigma^{\beta _1} , {\rm triv})] \right\rangle  \cong \mathbb{Z}_{2}$.
 \end{itemize}

By part 2 of Proposition \ref{12_21},   $[(\sigma^{\beta_1} , {\rm triv})] =  [( {\rm triv}, \tau '  )] $ where $\tau '_t =\alpha '$.
Therefore
 $$
{\rm H}^{2}\left( kL,k^{G}\right) = \left\langle [(\sigma^{\beta _1} , {\rm triv})] \right\rangle   \times\left\langle [({\rm triv},\tau^{2^{n-1}} _{1,2})] \right\rangle \cong \left\langle \alpha ' \right\rangle\times \left\langle \alpha^{2^{n-1}} _{1,2} \right\rangle\cong \mathbb{Z}_{2} \times\mathbb{Z}_{2}.
$$

\end{proof}

\subsection{Case 12: $t\rightharpoonup x = x^{-1}y^{2^{n-1}}, \quad
t\rightharpoonup y = y^{-1}$}\quad

 \begin{proposition}\label{-12-1} In this case
 $$
 {\rm H}^{2}\left( kL,k^{G}\right)\cong {\rm Ker}  \Psi \times {\rm Im} \Psi \cong \left\langle \alpha^{2^{n-1}} _{1,2} \right\rangle \times\left\langle \alpha^{2^{n-1}} _{2,2} \right\rangle  \cong \mathbb{Z}_{2} \times\mathbb{Z}_{2}.
 $$
\end{proposition}
\begin{proof}
Since $t\rightharpoonup x = x^{-1}y^{2^{n-1}}$ and $
t\rightharpoonup y = y^{-1}$,
\begin{eqnarray*}
t\rightharpoonup \alpha _{1,2} &=& \alpha _{1,1}^{2^{n-1}}\alpha _{1,2} \\
t\rightharpoonup \beta_{1}&=& \beta_{1}^{-1} \\
t\rightharpoonup \beta_{2}&=& \beta_{1}^{2^{n-1}}\beta_{2}^{-1}
\end{eqnarray*}
Thus
\begin{itemize}
\item ${\rm Im} \Psi = M\left( G\right)^{-} = \left\langle [\alpha^{2^{n-1}} _{1,2}] \right\rangle$.
\item $\hat G ^L = \left\langle \beta^{2^{n-1}} _1 \right\rangle \times \left\langle \beta^{2^{n-1}} _2 \right\rangle $ and $ N_t\left(\hat G\right) = \left\langle \beta^{2^{n-1}} _1 \right\rangle $.
\item ${\rm Ker}  \Psi = \left\langle [(\sigma^{\beta^{2^{n-1}} _2} , {\rm triv})] \right\rangle  \cong \mathbb{Z}_{2}$.
 \end{itemize}
 By part 4 of Proposition \ref{12_21},   $[(\sigma^{\beta^{2^{n-1}} _2} , {\rm triv})] =  [( {\rm triv}, \tau ^{2^{n-1}}_{2,2} )] $. Thus
$$
{\rm Ker}  \Psi = \left\langle [( {\rm triv}, \tau ^{2^{n-1}}_{2,2} )]  \right\rangle  \cong
\left\langle \alpha^{2^{n-1}} _{2,2} \right\rangle
$$
and
$$
 {\rm H}^{2}\left( kL,k^{G}\right) \cong
 {\rm H}_{c}^{2}\left( kL,k^{G}\right) \cong \left\langle \alpha^{2^{n-1}} _{2,2} \right\rangle \times\left\langle \alpha^{2^{n-1}} _{1,2} \right\rangle  \cong \mathbb{Z}_{2} \times\mathbb{Z}_{2} .
 $$
\end{proof}

\subsection{Case 13: $t\rightharpoonup x = x^{1+2^{n-1}}y^{2^{n-1}}, \quad
t\rightharpoonup y = y^{1+2^{n-1}}$}\quad

In this case we assume that $n\geq 3$, since for $n=2$ the matrix corresponding to the action of $t$ is equal to the matrix from Case 12.
\begin{proposition} In this case
 $$
 {\rm H}^{2}\left( kL,k^{G}\right) \cong {\rm Im} \Psi\cong \left\langle \alpha^{2^{n-1}} _{1,2} \right\rangle  \cong \mathbb{Z}_{2} .
 $$
\end{proposition}
\begin{proof}
Since $t\rightharpoonup x = x^{1+2^{n-1}}y^{2^{n-1}}$ and $
t\rightharpoonup y = y^{1+2^{n-1}}$,
\begin{eqnarray*}
\left( t\rightharpoonup\alpha _{1,2}\right)\left( x^{a_{1}}y^{a_{2}},x^{b_{1}}y^{b_{2}}\right) &=& \alpha _{1,2}\left( x^{\left( 1+2^{n-1}\right)a_{1}}y^{2^{n-1}a_{1}+\left( 1+2^{n-1}\right)a_{2}},x^{\left( 1+2^{n-1}\right)b_{1}}y^{2^{n-1}b_{1}+\left( 1+2^{n-1}\right)b_{2}}\right) \\
&=&\xi^{\left( 1+2^{n-1}\right)a_{1}(2^{n-1}b_{1}+\left( 1+2^{n-1}\right)b_{2})} =\xi^{a_{1}(2^{n-1}b_{1}+b_{2})}\\
\left( t\rightharpoonup \beta_{1}\right)\left( x^{a_{1}}y^{a_{2}}\right) &=& \beta_{1}\left( x^{\left( 1+2^{n-1}\right)a_{1}}y^{2^{n-1}a_{1}+\left( 1+2^{n-1}\right)a_{2}}\right)=\xi^{ \left( 1+2^{n-1}\right)a_{1}}\\
\left( t\rightharpoonup \beta_{2}\right)\left( x^{a_{1}}y^{a_{2}}\right) &=& \beta_{2}\left( x^{\left( 1+2^{n-1}\right)a_{1}}y^{2^{n-1}a_{1}+\left( 1+2^{n-1}\right)a_{2}}\right)=\xi^{2^{n-1}a_{1}+\left( 1+2^{n-1}\right)a_{2}}
\end{eqnarray*}

Therefore
\begin{eqnarray*}
t\rightharpoonup \alpha _{1,2} &=& \alpha _{1,1}^{2^{n-1}}\alpha _{1,2} \\
t\rightharpoonup \beta_{1}&=& \beta_{1}^{1+2^{n-1}}\\
t\rightharpoonup \beta_{2}&=& \beta_{1}^{2^{n-1}}\beta_{2}^{1+2^{n-1}} 
\end{eqnarray*}

Then 
\begin{itemize}
\item ${\rm Im} \Psi = M\left( G\right)^{-} = \left\langle [\alpha^{2^{n-1}} _{1,2}] \right\rangle $.
\item $\hat G ^L = \left\langle \beta^2 _1 \right\rangle \times \left\langle \beta^2 _2 \right\rangle $ and $ N_t\left(\hat G\right) = \left\langle \beta^2 _1 \right\rangle \times \left\langle \beta^2 _2 \right\rangle $.
\item ${\rm Ker}  \Psi = \{1\}$.
\item $
 {\rm H}^{2}\left( kL,k^{G}\right) \cong
 {\rm H}_{c}^{2}\left( kL,k^{G}\right) \cong \left\langle \alpha^{2^{n-1}} _{1,2} \right\rangle  \cong \mathbb{Z}_{2}
 $.
\end{itemize}
\end{proof}
\subsection{Case 14: $t\rightharpoonup x = x^{-1+2^{n-1}}y^{2^{n-1}}, \quad
t\rightharpoonup y = y^{-1+2^{n-1}}$}\quad

In this case we assume that $n\geq 3$, since for $n=2$ the matrix corresponding to the action of $t$ is equal to the matrix from Case 11.
\begin{proposition} In this case
 $$
 {\rm H}^{2}\left( kL,k^{G}\right) \cong {\rm Im} \Psi\cong \left\langle \alpha^{2^{n-1}} _{1,2} \right\rangle  \cong \mathbb{Z}_{2} .
 $$
\end{proposition}
\begin{proof}
Since $t\rightharpoonup x = x^{-1+2^{n-1}}y^{2^{n-1}} $ and $
t\rightharpoonup y = y^{-1+2^{n-1}}$,
\begin{eqnarray*}
\left( t\rightharpoonup\alpha _{1,2}\right)\left( x^{a_{1}}y^{a_{2}},x^{b_{1}}y^{b_{2}}\right) &=& \alpha _{1,2}\left( x^{\left( -1+2^{n-1}\right)a_{1}}y^{2^{n-1}a_{1}+\left( -1+2^{n-1}\right)a_{2}},x^{\left( -1+2^{n-1}\right)b_{1}}y^{2^{n-1}b_{1}+\left( -1+2^{n-1}\right)b_{2}}\right)\\
 &=&\xi^{\left( -1+2^{n-1}\right)a_{1}(2^{n-1}b_{1}+\left( -1+2^{n-1}\right)b_{2})} =\xi^{a_{1}(2^{n-1}b_{1}+b_{2})}\\
\left( t\rightharpoonup \beta_{1}\right)\left( x^{a_{1}}y^{a_{2}}\right) &=& \beta_{1}\left( x^{\left( -1+2^{n-1}\right)a_{1}}y^{2^{n-1}a_{1}+\left( -1+2^{n-1}\right)a_{2}}\right)=\xi^{ \left( -1+2^{n-1}\right)a_{1}}\\
\left( t\rightharpoonup \beta_{2}\right)\left( x^{a_{1}}y^{a_{2}}\right) &=& \beta_{2}\left( x^{\left( -1+2^{n-1}\right)a_{1}}y^{2^{n-1}a_{1}+\left( -1+2^{n-1}\right)a_{2}}\right)=\xi^{2^{n-1}a_{1}+\left( -1+2^{n-1}\right)a_{2}}
\end{eqnarray*}

Therefore
\begin{eqnarray*}
t\rightharpoonup \alpha _{1,2} &=& \alpha _{1,1}^{2^{n-1}}\alpha _{1,2} \\
t\rightharpoonup \beta_{1}&=& \beta_{1}^{-1+2^{n-1}}\\
t\rightharpoonup \beta_{2}&=& \beta_{1}^{2^{n-1}}\beta_{2}^{-1+2^{n-1}} 
\end{eqnarray*}

Then 
\begin{itemize}
\item ${\rm Im} \Psi = M\left( G\right)^{-} =\left\langle [\alpha^{2^{n-1}} _{1,2}] \right\rangle $.
\item $\hat G ^L = \left\langle \beta^{2^{n-1}} _1 \right\rangle \times \left\langle \beta^{2^{n-1}} _2 \right\rangle $ and $ N_t\left(\hat G\right) = \left\langle \beta^{2^{n-1}} _1 \right\rangle \times \left\langle \beta^{2^{n-1}} _2 \right\rangle $.
\item ${\rm Ker}  \Psi = \{1\}$.
\item $
 {\rm H}^{2}\left( kL,k^{G}\right) \cong
 {\rm H}_{c}^{2}\left( kL,k^{G}\right) \cong \left\langle \alpha^{2^{n-1}} _{1,2} \right\rangle  \cong \mathbb{Z}_{2}
 $.
\end{itemize}

\end{proof}


\subsection{Case 15: $t\rightharpoonup x = xy^{2^{n-1}}, \quad
t\rightharpoonup y = x^{2^{n-1}}y^{1+2^{n-1}}$}\quad

\begin{proposition} \label{122-1}
\begin{enumerate}
\item When $n \geq 3$,
 $$
 {\rm H}^{2}\left( kL,k^{G}\right)\cong  {\rm Im} \Psi \cong \left\langle \alpha^{2^{n-1}} _{1,2} \right\rangle  \cong \mathbb{Z}_{2} .
 $$
 \item When $n=2$,
$$
 {\rm H}^{2}\left( kL,k^{G}\right) \cong {\rm Im} \Psi\cong \left\langle \alpha _{1,1}\alpha _{1,2}\alpha  _{2,2} \right\rangle  \cong \mathbb{Z}_{4}
$$
\end{enumerate}
\end{proposition}
\begin{proof}
Since $t\rightharpoonup x = xy^{2^{n-1}}$ and $
t\rightharpoonup y = x^{2^{n-1}}y^{1+2^{n-1}}$,
\begin{eqnarray*}
\left( t\rightharpoonup\alpha _{1,2}\right)\left( x^{a_{1}}y^{a_{2}},x^{b_{1}}y^{b_{2}}\right) &=& \alpha _{1,2}\left( x^{a_{1}+2^{n-1}a_{2}}y^{2^{n-1}a_{1}+a_2+2^{n-1}a_{2}},x^{b_{1}+2^{n-1}b_{2}}y^{2^{n-1}b_{1}+b_2+2^{n-1}b_{2}}\right) \\
&=& \xi^{ (a_{1}+2^{n-1}a_{2})(2^{n-1}b_{1}+b_2+2^{n-1}b_{2})}=\xi^{ a_{1}b_{2}}(-1)^{a_{1}b_{2}+a_{1}b_{1}+a_{2}b_{2}}\\
\left( t\rightharpoonup\alpha _{1,1}\right)\left( x^{a_{1}}y^{a_{2}},x^{b_{1}}y^{b_{2}}\right) &=& \alpha _{1,1}\left( x^{a_{1}+2^{n-1}a_{2}}y^{2^{n-1}a_{1}+a_2+2^{n-1}a_{2}},x^{b_{1}+2^{n-1}b_{2}}y^{2^{n-1}b_{1}+b_2+2^{n-1}b_{2}}\right) \\
&=& \xi^{ (a_{1}+2^{n-1}a_{2})(b_{1}+2^{n-1}b_{2})}=\xi^{ a_{1}b_{1}}(-1)^{a_{1}b_{2}+a_{2}b_{1}}\\
\left( t\rightharpoonup\alpha _{2,2}\right)\left( x^{a_{1}}y^{a_{2}},x^{b_{1}}y^{b_{2}}\right) &=& \alpha _{2,2}\left( x^{a_{1}+2^{n-1}a_{2}}y^{2^{n-1}a_{1}+a_2+2^{n-1}a_{2}},x^{b_{1}+2^{n-1}b_{2}}y^{2^{n-1}b_{1}+b_2+2^{n-1}b_{2}}\right)\\
 &=& \xi^{ (2^{n-1}a_{1}+a_2+2^{n-1}a_{2})(2^{n-1}b_{1}+b_2+2^{n-1}b_{2})}
=\xi^{ a_{2}b_{2}}(-1)^{a_{1}b_{2}+a_{2}b_{1}}\\
\left( t\rightharpoonup \beta_{1}\right)\left( x^{a_{1}}y^{a_{2}}\right) &=& \beta_{1}\left( x^{a_{1}+2^{n-1}a_{2}}y^{2^{n-1}a_{1}+a_2+2^{n-1}a_{2}}\right)=\xi^{ a_{1}+2^{n-1}a_{2}}\\
\left( t\rightharpoonup \beta_{2}\right)\left( x^{a_{1}}y^{a_{2}}\right) &=& \beta_{2}\left(x^{a_{1}+2^{n-1}a_{2}}y^{2^{n-1}a_{1}+a_2+2^{n-1}a_{2}}\right)=\xi^{2^{n-1}a_{1}+a_2+2^{n-1}a_{2}}
\end{eqnarray*}

Therefore
\begin{eqnarray*}
t\rightharpoonup \alpha _{1,2} &=& \alpha^{1+2^{n-1}} _{1,2} \alpha^{2^{n-1}} _{1,1}\alpha^{2^{n-1}} _{2,2} \\
t\rightharpoonup \alpha _{j,j} &=& \alpha _{j,j}\alpha^2 _{1,2}\alpha^2 _{2,1} \text{ for } j=1,2 \\
t\rightharpoonup \beta_{1}&=& \beta_{1}\beta^{2^{n-1}}_{2} \\
t\rightharpoonup \beta_{2}&=& \beta^{2^{n-1}}_{1}\beta^{1+2^{n-1}}_{2}
\end{eqnarray*}
Thus
\begin{itemize}
\item $\hat G ^L = \left\langle \beta^2 _1 \right\rangle \times \left\langle \beta^2 _2 \right\rangle  $ and $ N_t\left(\hat G\right) = \left\langle \beta^2 _1 \right\rangle \times \left\langle \beta^2 _2 \right\rangle $.
\item ${\rm Ker}  \Psi = \{ 1\}$
 \end{itemize}
\begin{enumerate}
\item Assume first that $n \geq 3$. Then ${\rm Im} \Psi =M\left( G\right)^- = \left\langle [\alpha ^{2^{n-1}} _{1,2}] \right\rangle$ and
$$
{\rm H}^{2}\left( kL,k^{G}\right) \cong {\rm H}^{2}_c\left( kL,k^{G}\right)\cong {\rm Im} \Psi\cong \left\langle \alpha ^{2^{n-1}} _{1,2} \right\rangle \cong \mathbb{Z}_{2}.
$$
\item Let us now consider the case of $n=2$. We will first show that  ${\rm Im} \Psi = M\left( G\right)$.

By part 1 of Proposition \ref{12_21}, $\left[ \alpha _{1,2}\alpha _{1,1}\alpha _{2,2}\right]=\left[ \alpha _{1,2}\right] $ in $M\left( G\right)$.
Moreover,
$$
\left(\alpha _{1,2}\alpha _{1,1}\alpha _{2,2} \right) \left( t\rightharpoonup \left(\alpha _{1,2}\alpha _{1,1}\alpha _{2,2} \right) \right) =
\alpha _{1,2}\alpha _{1,1}\alpha _{2,2}\alpha^{-1} _{1,2} \alpha^{2} _{1,1}\alpha^{2} _{2,2}\alpha _{1,1}\alpha^2 _{1,2}\alpha^2 _{2,1}\alpha _{2,2}\alpha^2 _{1,2}\alpha^2 _{2,1}=1
$$
Therefore $\left[ \alpha _{1,2}\right] =\left[ \alpha _{1,2}\alpha _{1,1}\alpha _{2,2}\right] \in {\rm Im} \Psi$ and
$$
 {\rm H}^{2}\left( kL,k^{G}\right)  \cong {\rm H}^{2}_c\left( kL,k^{G}\right)\cong {\rm Im} \Psi\cong \left\langle \alpha _{1,1}\alpha _{1,2}\alpha  _{2,2} \right\rangle  \cong \mathbb{Z}_{4}
$$
\end{enumerate}
\end{proof}


\subsection{Case 16: $t\rightharpoonup x = x^{-1}y^{2^{n-1}}, \quad
t\rightharpoonup y = x^{2^{n-1}}y^{-1+2^{n-1}}$}\quad

In this case we assume that $n\geq 3$, since for $n=2$ the matrix corresponding to the action of $t$ is similar to the matrix from Case 15.
\begin{proposition} In this case
 $$
 {\rm H}^{2}\left( kL,k^{G}\right) \cong{\rm Im} \Psi\cong \left\langle \alpha^{2^{n-1}} _{1,2} \right\rangle  \cong \mathbb{Z}_{2} .
 $$
\end{proposition}
\begin{proof}
Since $t\rightharpoonup x = x^{-1}y^{2^{n-1}}$ and $
t\rightharpoonup y = x^{2^{n-1}}y^{-1+2^{n-1}}$,
\begin{eqnarray*}
\left( t\rightharpoonup\alpha _{1,2}\right)\left( x^{a_{1}}y^{a_{2}},x^{b_{1}}y^{b_{2}}\right) &=& \alpha _{1,2}\left( x^{-a_{1}+2^{n-1}a_{2}}y^{2^{n-1}a_{1}-a_2+2^{n-1}a_{2}},x^{-b_{1}+2^{n-1}b_{2}}y^{2^{n-1}b_{1}-b_2+2^{n-1}b_{2}}\right) \\
&=& \xi^{ (-a_{1}+2^{n-1}a_{2})(2^{n-1}b_{1}-b_2+2^{n-1}b_{2})}=\xi^{ a_{1}b_{2}}(-1)^{a_{1}b_{2}+a_{1}b_{1}+a_{2}b_{2}}\\
\left( t\rightharpoonup \beta_{1}\right)\left( x^{a_{1}}y^{a_{2}}\right) &=& \beta_{1}\left( x^{-a_{1}+2^{n-1}a_{2}}y^{2^{n-1}a_{1}-a_2+2^{n-1}a_{2}}\right)=\xi^{ -a_{1}+2^{n-1}a_{2}}\\
\left( t\rightharpoonup \beta_{2}\right)\left( x^{a_{1}}y^{a_{2}}\right) &=& \beta_{2}\left(x^{-a_{1}+2^{n-1}a_{2}}y^{2^{n-1}a_{1}-a_2+2^{n-1}a_{2}}\right)=\xi^{2^{n-1}a_{1}-a_2+2^{n-1}a_{2}}
\end{eqnarray*}

Therefore
\begin{eqnarray*}
t\rightharpoonup \alpha _{1,2} &=& \alpha^{1+2^{n-1}} _{1,2} \alpha^{2^{n-1}} _{1,1}\alpha^{2^{n-1}} _{2,2} \\
t\rightharpoonup \beta_{1}&=& \beta^{-1}_{1}\beta^{2^{n-1}}_{2} \\
t\rightharpoonup \beta_{2}&=& \beta^{2^{n-1}}_{1}\beta^{-1+2^{n-1}}_{2}
\end{eqnarray*}
\begin{itemize}
\item Since $n \geq 3$, ${\rm Im} \Psi = M\left( G\right) ^-= \left\langle [\alpha ^{2^{n-1}} _{1,2}] \right\rangle$.
\item $\hat G ^L = \left\langle \beta^{2^{n-1}} _1 \right\rangle \times \left\langle \beta^{2^{n-1}} _2 \right\rangle  $ and $ N_t\left(\hat G\right) = \left\langle \beta^{2^{n-1}} _1 \right\rangle \times \left\langle \beta^{2^{n-1}} _2 \right\rangle $.
\item ${\rm Ker}  \Psi = \{ 1\}$

\item ${\rm H}^{2}\left( kL,k^{G}\right) \cong {\rm H}^{2}_c\left( kL,k^{G}\right)\cong \left\langle \alpha ^{2^{n-1}} _{1,2} \right\rangle \cong \mathbb{Z}_{2} $.
 \end{itemize}
\end{proof}
\section{Hopf algebras of dimension $32$ with $\mathbf{G}\left(
H\right) =\mathbb{Z}_{4}\times \mathbb{Z}_{4}$}\label{32}

In this section we describe all possible nontrivial semisimple
Hopf algebras of dimension $32$ with a group of group-like elements isomorphic to $%
\mathbb{Z}_{4}\times \mathbb{Z}_{4}$. We also compute their groups of
central group-like elements, the groups of group-like elements of their duals, and their
irreducible representations.

Let $H$ be a nontrivial semisimple
Hopf algebra of dimension $32$ whose group of group-like elements isomorphic to $%
\mathbb{Z}_{4}\times \mathbb{Z}_{4}$. As in Section \ref{main}, $H$ is equivalent to the bicrossed product
$K\#_{\sigma }^{\tau }F$ of $K=k^G=\left( kG \right) ^{\ast }$ where $G =\widehat{\mathbf{G}\left( H\right) }\cong \mathbb{Z}_{4}\times \mathbb{Z}_{4}$ and $%
F=kL=k\left\langle t\right\rangle \cong k\mathbb{Z}_{2}$ with an action $%
\rightharpoonup :F\otimes K\rightarrow K$, trivial coaction, a cocycle $%
\sigma :F\otimes F\rightarrow K$ and a dual cocycle $\tau
:F\rightarrow K\otimes K$, and the action by $t$ being an order $2$
Hopf algebra automorphism of $K$. Write
\begin{eqnarray*}
\sigma (t,t) &=& \sum\limits_{g\in
G}\sigma_t(g) p_g\\
\tau (t) &=& \sum_{g,h \in G} \tau_{t}(g,h) p_g\otimes p_h
\end{eqnarray*}
Denote such a Hopf algebra by $H\cong
K\#_{\sigma }^{\tau }F$ by $H\left( \rightharpoonup , \sigma _t, \tau _t\right)$. When $\sigma$ is trivial, we will denote $H$ by $H\left( \rightharpoonup , \tau _t\right)$. By \cite[Theorem 4.4 (1)]{Mas} (or Remark \ref{trivial}), we may assume that every Hopf algebra $H$ under consideration is equivalent to $K\#^{\tau }F$ with the trivial cocycle, but for some computations it would be beneficial to consider the representative of the equivalence class of $H$  with non-trivial cocycle $\sigma$ but with simpler dual cocycle $\tau$ (as it was done in \cite{Ka1} and \cite{Ka2}).

Elements $p_g \# t^j$, where $g\in G$, $j=0,1$, form the basis of of $H$. However, when $\sigma$ is assumed to be trivial, we will just denote them by $p_g t^j$.

As in Section \ref{main} we define
  \begin{eqnarray*}
    \beta_{1}\left( x^{a}y^{b}\right) &=&i^{ a}\\
     \beta_{2}\left( x^{a}y^{b}\right) &=&i^{ b}\\
\alpha _{p,q}\left( x^{a_{1}}y^{a_{2}},x^{b_{1}}y^{b_{2}}\right) &=&i
^{a_{p}b_{q}}\\
\left(\tau _{p,q}\right)_t &=&\alpha _{p,q}
\end{eqnarray*}

For all of the Hopf algebras which we describe, we are going to compute their irreducible representations and try to distinguish these Hopf algebras based on their Grothendieck ring structure. To do that we will compute the groups of group-like elements of their dual Hopf algebras and the squares of their irreducible representations. If the Hopf algebras cannot be distinguished based on this information, we will compute their Frobenius-Schur indicators.

In \cite{KMM} it was shown that for $H\cong
\left( kG \right) ^{\ast }\#_{\sigma }^{\tau }kL$ all simple left $H$-modules may be described
as induced modules from certain twisted group algebras of the stabilizers $%
L_{g} =\{y\in
L|y\rightharpoonup g=g\}$. This approach follows the general outline of the work on more
general crossed products $B=A\#_{\sigma }kL$, for any semisimple algebra $A$%
, in \cite{MW} and \cite{W}. However in \cite{KMM} we got more explicit
information more easily, because of the special nature of our algebra $%
A=(kG)^{\ast }$ and of the $L$-action on $A$.

Since $L$ has order $2$, then $H$ has only irreducible representations of degrees $1$ and $2$ and $\left\vert \mathbf{G}\left( {\rm H}^{\ast }\right)
\right\vert =2\left\vert G^{L}\right\vert $ (see \cite[Proposition 3.7]{Ka3}). In particular, when we assume that $\sigma$ is trivial,
$\mathbf{G}\left( {\rm H}^{\ast }\right)
 =\chi_g \chi_t^j | g\in G^L \text{ and } j=0,1$  where
\begin{eqnarray*}
\chi_g (p_h) &=& \delta _{g,h} \qquad \chi_g (t) =1\\
\chi_t (p_h) &=& \delta _{1,h} \qquad \chi_g (t) =-1
\end{eqnarray*}

We will use \cite{Ka2} for an explicit description of irreducible representations of degree $2$. By \cite[Proposition 3.4]{Ka2}, for $H\cong
\left( kG \right) ^{\ast }\#^{\tau }kL$  every irreducible representation $\pi$ of degree $2$ is induced from some irreducible representation $\psi$ of $\left( kG \right) ^{\ast }$. By \cite[Remark 2.5]{Ka2}
\begin{eqnarray*}
\pi \left( p_g\right) &=&\left(
\begin{array}{cc}
\psi \left( p_g\right) & 0  \\
0 & \psi \left( p_t\rightharpoonup g\right)
\end{array}
\right)\\
\pi \left( t\right) &=&\left(
\begin{array}{cc}
0 & 1  \\
1 & 0
\end{array}
\right)
\end{eqnarray*}
This representation is irreducible if and only if $g \notin G^L$. Thus $H\cong
\left( kG \right) ^{\ast }\#^{\tau }kL$ has exactly $\frac{|G|-\left\vert G^{L}\right\vert}{2}$ irreducible $2$-dimensional representations $\pi_{g, t\rightharpoonup g}$ (where $g$ ranges over a choice of one element
in each two-element $L$-orbit of $G$) defined by
\begin{eqnarray*}
\pi_{g, t\rightharpoonup g} \left( p_h\right) &=&\left(
\begin{array}{cc}
\delta _{g,h} & 0  \\
0 & \delta _{g,t\rightharpoonup h}
\end{array}
\right)\\
\pi_{g, t\rightharpoonup g} \left( t\right) &=&\left(
\begin{array}{cc}
0 & 1  \\
1 & 0
\end{array}
\right)
\end{eqnarray*}

Let $\chi_{g, t\rightharpoonup g}$ be the character of $\pi_{g, t\rightharpoonup g}$, a  $2$-dimensional representation of $H\cong
\left( kG \right) ^{\ast }\#^{\tau }kL$. Applying \cite[Corollary 4.12]{KMM} to compute the second Frobenius-Schur indicator $\nu_2 \left(\chi_{g,t\rightharpoonup g}\right)$, we get the following formula:

\begin{equation}\label{nu_2}
\nu_2 \left(\chi_{g,t\rightharpoonup g}\right)= \left\{
\begin{array}{cl}%
1, & \text{ if } g^2=1 \\
\tau _t (g,g^{-1}), & \text{ if } g^2\neq 1 \text{ and } t\rightharpoonup g = g^{-1}\\
0 & \text{ otherwise}
\end{array} \right.
\end{equation}

In Theorem \ref{action} we have described all nonsimilar actions of $t$ on $G$ which correspond to $2 \times 2$ matrices of order $2$. In the case of $n=2$ (that is, when $%
\mathbf{G}\left( H\right) \cong \mathbb{Z}_{4}\times \mathbb{Z}_{4}$), there are $27$ matrices of this type which are partitioned into six conjugacy classes with the following representatives and sizes:
\begin{enumerate}
\item $\left[
\begin{array}{ll}
-1 & 0 \\
0 & -1
\end{array}%
\right]$ (of size $1$);
\item $\left[
\begin{array}{ll}
1 & 0 \\
0 & -1
\end{array}%
\right]$ (of size $6$);
\item $\left[
\begin{array}{ll}
1 & 0 \\
2 & 1
\end{array}%
\right]$ (of size $3$);
\item $\left[
\begin{array}{ll}
-1 & 0 \\
2 & -1
\end{array}%
\right]$ (of size $3$);
\item $\left[
\begin{array}{ll}
1 & 2 \\
2 & -1
\end{array}%
\right]$ (of size $2$);
\item $\left[
\begin{array}{ll}
0 & 1 \\
1 & 0
\end{array}%
\right]$ (of size $12$);
\end{enumerate}

Therefore every Hopf algebra under consideration will be described in one of the following six propositions:

\begin{proposition}
Assume that $$t\rightharpoonup_1 x = x^{-1}, \quad t\rightharpoonup_1 y = y^{-1}.$$ Then  $H\cong H\left(\rightharpoonup_1, \alpha^2 _{1,2}\right)$ or $ H\left(\rightharpoonup_1, \alpha^2_{1,1}\alpha^2_{1,2}\alpha^2_{2,2}\right)$. Moreover
\begin{enumerate}
\item $\mathbf{G}\left( H\right) \cap Z(H) = \hat G ^L = \left\langle \beta^2 _1 \right\rangle \times \left\langle \beta^2 _2 \right\rangle \cong \mathbb{Z}_{2} \times\mathbb{Z}_{2} $.
\item $\mathbf{G}\left( H^*\right) =\left\langle a \right\rangle \times \left\langle b \right\rangle\times \left\langle c \right\rangle\cong \mathbb{Z}_{2} \times\mathbb{Z}_{2} \times \mathbb{Z}_{2}$.
\item $H$ has $6$ two-dimensional representations $\pi_1 , \ldots ,\pi _6$ with characters  $\chi_1 , \ldots ,\chi _6$ such that
\begin{eqnarray*}
\pi_1 ^2 &=&\pi_2 ^2 = 1+ c +a +ac\\
\pi_3 ^2 &=&\pi_4 ^2 = 1+ c +b +bc\\
\pi_5 ^2 &=&\pi_6 ^2 = 1+ c +ab +abc
\end{eqnarray*}
\item When $H\cong H\left(\rightharpoonup_1, \alpha^2 _{1,2}\right)$,
$$
\nu_2 (\chi _j)=\left\{\begin{array}{cl}%
1, & j=1,2,3,4 \\
-1, & j=5,6
\end{array} \right .
$$
\item When $H\cong H\left(\rightharpoonup_1, \alpha^2_{1,1}\alpha^2_{1,2}\alpha^2_{2,2}\right)$,
$$
\nu_2 (\chi _j)=-1\quad \text{for }  j=1,2,3,4,5,6
$$
\end{enumerate}
\end{proposition}
\begin{proof}
By Proposition \ref{-1-1}, up to equivalence, there are eight Hopf algebras of this type. Four of them, $ H\left(\rightharpoonup_1, \alpha^2_{1,2}\right)$, $ H\left(\rightharpoonup_1, \alpha^2_{1,1}\alpha^2_{1,2}\right)$, $ H\left(\rightharpoonup_1, \alpha^2_{1,2}\alpha^2_{2,2}\right)$, and $ H\left(\rightharpoonup_1, \alpha^2_{1,1}\alpha^2_{1,2}\alpha^2_{2,2}\right)$ are non-trivial.  Define
\begin{eqnarray*}
f_1 :G
&\to & G \text{ via} \\
f_1\left( x\right) &=&xy\\
f_1\left( y\right) &=&y
\end{eqnarray*}
and
\begin{eqnarray*}
f_2:G
&\to & G \text{ via} \\
f_2\left( x\right) &=&x\\
f_2\left( y\right) &=&xy
\end{eqnarray*}
Then $f_j ( t\rightharpoonup_1 g) =t\rightharpoonup_1 f_j (  g) $  and
\begin{eqnarray*}
\alpha^2_{1,2} \circ (f_1 \times f_1) \left( x^{a_{1}}y^{a_{2}},x^{b_{1}}y^{b_{2}}\right) &=& \alpha^2_{1,2}\left( x^{a_{1}}y^{a_{1}+a_{2}},x^{b_{1}}y^{b_{1}+b_{2}}\right) =(-1)^{a_{1}\left( b_{1}+b_{2}\right)}\\
\alpha^2_{1,2} \circ (f_2 \times f_2) \left( x^{a_{1}}y^{a_{2}},x^{b_{1}}y^{b_{2}}\right) &=& \alpha^2_{1,2}\left( x^{a_{1}+a_{2}}y^{a_{2}},x^{b_{1}+b_{2}}y^{b_{2}}\right) =(-1)^{\left( a_{1}+a_{2}\right)b_{2}}
\end{eqnarray*}
and therefore
\begin{eqnarray*}
\alpha^2_{1,2} \circ (f_1 \times f_1) &=& \alpha^2_{1,1}\alpha^2_{1,2}\\
\alpha^2_{1,2} \circ (f_2 \times f_2)  &=& \alpha^2_{1,2} \alpha^2_{2,2}
\end{eqnarray*}
Therefore $H( \alpha^2_{j,j} \alpha ^2_{1,2})\cong  H(  \alpha ^2_{1,2})$ for $j=1,2$.

Let $H\cong H\left(\rightharpoonup_1, \alpha^2 _{1,2}\right)$ or $ H\left(\rightharpoonup_1, \alpha^2_{1,1}\alpha^2_{1,2}\alpha^2_{2,2}\right)$. Then $H$ has $8$ one-dimensional characters of the type $\chi_g \chi_t^j$ for $j=0,1$ and $g\in G^L=\{1, x^2, y^2, x^2y^2\}$ where
\begin{eqnarray*}
\chi_g (p_h) &=& \delta _{g,h} \qquad \chi_g (t) =1\\
\chi_t (p_h) &=& \delta _{1,h} \qquad \chi_g (t) =-1
\end{eqnarray*}
For each one-dimensional character $\chi$, $\chi^2 =1$. Set $a=\chi _{x^2}$, $b=\chi _{y^2}$ and $c=\chi _{t}$. Then
$$
\mathbf{G}\left( H^*\right) =\left\langle a \right\rangle \times \left\langle b \right\rangle\times \left\langle c \right\rangle\cong \mathbb{Z}_{2} \times\mathbb{Z}_{2} \times \mathbb{Z}_{2}.
$$
$H$ has $6$ two-dimensional characters $\chi_{g,g^{-1}}=\chi_{g,t\rightharpoonup_1 g}$ for $g\in \{ x,xy^2,y,x^2y, xy, xy^{-1}\}$ where
$$
\chi_{g,g^{-1}} (p_h) = \delta _{g,h}+ \delta _{g^{-1},h} \qquad \chi_{g,g^{-1}} (p_h t) =0.
$$
Then
\begin{eqnarray*}
\chi_{x,x^{-1}} ^2 &=&\chi_{xy^2,x^{-1}y^2} ^2 = 1+ \chi _t +\chi _{x^2} +\chi _{x^2}\chi _t\\
\chi_{y,y^{-1}} ^2 &=&\chi_{x^2y,x^2y^{-1}} ^2 = 1+ \chi _t +\chi _{y^2} +\chi _{y^2}\chi _t\\
\chi_{xy,x^{-1}y^{-1}} ^2 &=&\chi_{xy^{-1},x^{-1}y} ^2 = 1+ \chi _t +\chi _{x^2y^2} +\chi _{x^2y^2}\chi _t
\end{eqnarray*}
Since $t\rightharpoonup_1 g = g^{-1}$, by formula (\ref{nu_2}) we have
$$
\nu_2 \left(\chi_{g,g^{-1}}\right)= \tau _t (g,g^{-1}) \quad \text{for all } g\in \{ x,xy^2,y,x^2y, xy, xy^{-1}\}
$$
Therefore, when $\tau _t = \alpha^2_{1,2}$,
\begin{eqnarray*}
\nu_2 \left(\chi_{x,x^{-1}}\right)=\nu_2 \left(\chi_{xy^2,x^{-1}y^2}\right)=\nu_2 \left(\chi_{y,y^{-1}}\right)=\nu_2 \left(\chi_{x^2y,x^2y^{-1}}\right)&=&1\\
\nu_2 \left(\chi_{xy,x^{-1}y^{-1}}\right)=\nu_2 \left(\chi_{xy^{-1},x^{-1}y}\right)&=&-1
\end{eqnarray*}
and when $\tau _t = \alpha^2_{1,1}\alpha^2_{1,2}\alpha^2_{2,2}$,
$$
\nu_2 \left(\chi_{g,g^{-1}}\right)=-1 \quad \text{for all } g\in \{ x,xy^2,y,x^2y, xy, xy^{-1}\}
$$
\end{proof}
\begin{proposition}
Assume that $$t\rightharpoonup_2 x = x, \quad t\rightharpoonup_2 y = y^{-1}.$$ Then  $H\cong H\left(\rightharpoonup_2,  \alpha _{1,2}\right)$, $H\left( \rightharpoonup_2,   \alpha ^2_{1,2}\right)$,  $H\left(\rightharpoonup_2,    \alpha^2_{1,2}\alpha^2_{2,2}\right)$, $H\left(\rightharpoonup_2,  \beta _1 , \alpha _{1,2} \right)$, $H\left(\rightharpoonup_2,  \beta _1 , \alpha _{1,2}\alpha ^{2}_{2,2} \right)$   or $H\left(\rightharpoonup_2,  \beta _1 , \alpha^2 _{1,2}\right)$. Moreover
\begin{enumerate}
\item $\mathbf{G}\left( H\right) \cap Z(H) = \hat G ^L = \left\langle \beta _1 \right\rangle \times \left\langle \beta^2 _2 \right\rangle \cong \mathbb{Z}_{4} \times\mathbb{Z}_{2} $.
\item $|\mathbf{G}\left( H^*\right)| =16$ and $H$ has $4$ two-dimensional representations $\pi_1 , \ldots ,\pi _4$ with characters  $\chi_1 , \ldots ,\chi _4$ such that
$$
\nu_2 (\chi _3)=\nu_2 (\chi _4)=0
$$

\item When $H\cong H\left(\rightharpoonup_2, \alpha _{1,2}\right)$,
$\mathbf{G}\left( H^*\right) =\left\langle a, b , c |a^4=b^2=c^2=1, bab=ac\right\rangle\cong G_6$ and
\begin{eqnarray*}
\pi_1 ^2 &=&\pi_2 ^2 = 1+ c +b +bc\\
\pi_3 ^2 &=&\pi_4 ^2 = a^2+ a^2c +a^2b +a^2bc\\
\nu_2 (\chi _1)&=&1 \quad \text{ and } \quad\nu_2 (\chi _2)=-1.
\end{eqnarray*}

\item When $H\cong H\left(\rightharpoonup_2, \alpha^2_{1,2}\right)$, $\mathbf{G}\left( H^*\right) =\left\langle a \right\rangle \times \left\langle b \right\rangle\times \left\langle c \right\rangle\cong \mathbb{Z}_{4} \times\mathbb{Z}_{2} \times \mathbb{Z}_{2}$ and
\begin{eqnarray*}
\pi_1 ^2 &=&\pi_2 ^2 = 1+ c +b +bc\\
\pi_3 ^2 &=&\pi_4 ^2 = a^2+ a^2c +a^2b +a^2bc\\
\nu_2 (\chi _1)&=&\nu_2 (\chi _2)=1.
\end{eqnarray*}

\item When $H\cong H\left(\rightharpoonup_2, \alpha^2_{1,2}\alpha^2_{2,2}\right)$, $\mathbf{G}\left( H^*\right) =\left\langle a \right\rangle \times \left\langle b \right\rangle\times \left\langle c \right\rangle\cong \mathbb{Z}_{4} \times\mathbb{Z}_{2} \times \mathbb{Z}_{2}$ and
\begin{eqnarray*}
\pi_1 ^2 &=&\pi_2 ^2 = 1+ c +b +bc\\
\pi_3 ^2 &=&\pi_4 ^2 = a^2+ a^2c +a^2b +a^2bc\\
\nu_2 (\chi _1)&=&\nu_2 (\chi _2)=-1.
\end{eqnarray*}
\item When $H\cong H\left(\rightharpoonup_2,\beta_1,\alpha _{1,2}\right)$, $
\mathbf{G}\left( H^*\right) =\left\langle a, b  |a^8=b^2=1, bab=a^5\right\rangle\cong G_1$ and
\begin{eqnarray*}
\pi_1 ^2 &=&\pi_2 ^2 = 1+ a^4 +b +a^4b\\
\pi_3 ^2 &=&\pi_4 ^2 = a^2+ a^6 +a^2b +a^6b\\
\nu_2 (\chi _1)&=&\nu_2 (\chi _2)=1.
\end{eqnarray*}
\item When $H\cong H\left(\rightharpoonup_2,\beta_1,\alpha _{1,2}\alpha^2_{2,2}\right)$, $
\mathbf{G}\left( H^*\right) =\left\langle a, b  |a^8=b^2=1, bab=a^5\right\rangle\cong G_1$ and
\begin{eqnarray*}
\pi_1 ^2 &=&\pi_2 ^2 = 1+ a^4 +b +a^4b\\
\pi_3 ^2 &=&\pi_4 ^2 = a^2+ a^6 +a^2b +a^6b\\
\nu_2 (\chi _1)&=&\nu_2 (\chi _2)=-1.
\end{eqnarray*}
\item When $H\cong H\left(\rightharpoonup_2,\beta_1,\alpha^2 _{1,2}\right)$, $
\mathbf{G}\left( H^*\right) =\left\langle a \right\rangle \times \left\langle b \right\rangle\cong \mathbb{Z}_{8} \times\mathbb{Z}_{2} $ and
\begin{eqnarray*}
\pi_1 ^2 &=&\pi_2 ^2 = 1+ a^4 +b +a^4b\\
\pi_3 ^2 &=&\pi_4 ^2 = a^2+ a^6 +a^2b +a^6b\\
\nu_2 (\chi _1)&=&\nu_2 (\chi _2)=1.
\end{eqnarray*}
\end{enumerate}

\end{proposition}
\begin{proof}
By Proposition \ref{1-1}, up to equivalence, there are sixteen Hopf algebras of this type. Twelve of them,
$ H\left(\rightharpoonup_2,\beta_1^j\alpha^k_{1,2}\alpha^l_{2,2}\right)$ for $j=0,1, k=1,2,3, l=0,2$, are non-trivial. Define
\begin{eqnarray*}
f_1 :G
&\to & G \text{ via} \\
f_1\left( x\right) &=&t\rightharpoonup x =x\\
f_1\left( y\right) &=&t\rightharpoonup y = y^{-1}
\end{eqnarray*}
and
\begin{eqnarray*}
f_2:G
&\to & G \text{ via} \\
f_2\left( x\right) &=&x\\
f_2\left( y\right) &=&x^2y
\end{eqnarray*}
and
\begin{eqnarray*}
f_3:G
&\to & G \text{ via} \\
f_3\left( x\right) &=&xy^2\\
f_3\left( y\right) &=&x^2y
\end{eqnarray*}
Then $f_j ( t\rightharpoonup_2 g) =t\rightharpoonup_2 f_j (  g) $  and
\begin{eqnarray*}
\alpha _{1,2} \circ (f_1 \times f_1) \left( x^{a_{1}}y^{a_{2}},x^{b_{1}}y^{b_{2}}\right) &=& i^{-a_{1}b_{2}}\\
\alpha_{1,2} \circ (f_2 \times f_2) \left( x^{a_{1}}y^{a_{2}},x^{b_{1}}y^{b_{2}}\right) &=& i^{a_{1}b_{2}}(-1)^{a_{2}b_{2}}\\
\alpha ^2_{1,2} \circ (f_3 \times f_3) \left( x^{a_{1}}y^{a_{2}},x^{b_{1}}y^{b_{2}}\right) &=& (-1)^{a_{1}b_{2}}\\
\beta _1 \circ f_1 \left( x^{a_{1}}y^{a_{2}}\right) &=& i^{a_{1}}\\
\beta _1 \circ f_3 \left( x^{a_{1}}y^{a_{2}}\right) &=& i^{a_{1}}(-1)^{a_{2}}\\
\alpha ^2_{2,2} \circ (f_j \times f_j) \left( x^{a_{1}}y^{a_{2}},x^{b_{1}}y^{b_{2}}\right)& =&(-1)^{a_{2}b_{2}}
\end{eqnarray*}
and therefore
\begin{eqnarray*}
\alpha _{1,2} \circ (f_1 \times f_1) &=& \alpha^{3}_{1,2}\\
\alpha _{1,2} \circ (f_2 \times f_2)  &=& \alpha _{1,2} \alpha^2_{2,2}\\
\alpha^2_{1,2} \circ (f_3 \times f_3)  &=& \alpha^2_{1,2} \\
\beta _1 \circ f_1 &=& \beta _1\\
\beta _1 \circ f_3 &=& \beta _1 \beta ^2 _2\\
\alpha^2_{2,2} \circ (f_j \times f_j) &=& \alpha^2_{2,2}
\end{eqnarray*}
Therefore
\begin{eqnarray*}
H\left(\rightharpoonup_2,\beta _1 ^j , \alpha ^{3}
 _{1,2}\alpha ^{l}_{2,2} \right)&\cong &H\left(\rightharpoonup_2,\beta _1 ^j, \alpha
 _{1,2}\alpha ^{l}_{2,2} \right) \text{ for } j=0,1,l=0,2\\
H \left(\rightharpoonup_2,\alpha _{1,2}\alpha ^{2}_{2,2}\right) &\cong& H \left(\rightharpoonup_2,\alpha _{1,2}\right)\\
H \left(\rightharpoonup_2,\beta _1 , \alpha ^2
 _{1,2}\alpha ^{2}_{2,2} \right)&\cong& H\left(\rightharpoonup_2,\beta _1 , \alpha ^{2}
 _{1,2} \right)
\end{eqnarray*}

For further computations we will use the fact that, by Proposition \ref{1-1},  $ H\left(\rightharpoonup_2,\beta_1,\alpha^j_{1,2}\alpha^l_{2,2}\right)$ is equivalent to $ H\left(\rightharpoonup_2,\alpha ' \alpha^j_{1,2}\alpha^l_{2,2}\right)$ for $\alpha '$ defined via
$$
\alpha ' \left( x^{a_{1}}y^{a_{2}},x^{b_{1}}y^{b_{2}}\right) =\left( -1\right)
^{k}
$$
where $k= 0$ if $0\leq a_1+b_1 < 4$ and $k = 1$ if $4\leq a_1+b_1 < 8$.

Let $H$ be a Hopf algebra corresponding to $\rightharpoonup_2$. Then $H$ has $16$ one-dimensional characters of the type $\chi_g \chi_t^j$ for $j=0,1$ and $g\in G^L=\left\langle x \right\rangle \times \left\langle y^2 \right\rangle$ where
\begin{eqnarray*}
\chi_g (p_h) &=& \delta _{g,h} \qquad \chi_g (t) =1\\
\chi_t (p_h) &=& \delta _{1,h} \qquad \chi_g (t) =-1
\end{eqnarray*}
Then
\begin{eqnarray*}
\chi_{y^2}^2  &=& \chi_t^2=\chi_1 =1\\
\chi_{x}^2 &=& \chi_{x^2}\\
\chi_t &\in& Z\left(\mathbf{G}\left( H^*\right)\right)
\end{eqnarray*}

$H$ has $4$ two-dimensional characters $\chi_{g,t\rightharpoonup_2 g}$ for $g\in \{ y, x^2y, xy, x^{-1}y\}$ where
$$
\chi_{g,t\rightharpoonup_2 g} (p_h) = \delta _{g,h}+ \delta _{t\rightharpoonup_2 g,h} \qquad \chi_{g,t\rightharpoonup_2 g} (p_h t) =0.
$$
Then
\begin{eqnarray*}
\chi_{y,y^{-1}} ^2 &=&\chi_{x^2y,x^2y^{-1}} ^2 = 1+ \chi _t +\chi _{y^2} +\chi _{y^2}\chi _t\\
\chi_{xy,x^{-1}y^{-1}} ^2 &=&\chi_{xy^{-1},x^{-1}y} ^2 = \chi _{x^2}+ \chi _{x^2}\chi _t +\chi _{x^2y^2} +\chi _{x^2y^2}\chi _t
\end{eqnarray*}
Then  by formula (\ref{nu_2})
\begin{eqnarray*}
\nu_2 \left(\chi_{g,t\rightharpoonup_2 g}\right)&=& \tau _t (g,g^{-1})\quad \text{for } g\in \{ y, x^{2}y\}\\
\nu_2 \left(\chi_{xy,x^{-1}y^{-1}}\right)&=& \nu_2 \left(\chi_{xy^{-1},x^{-1}y}\right)=0
\end{eqnarray*}
\begin{enumerate}
\item Assume that $H\cong H\left(\rightharpoonup_2, \alpha _{1,2}\right)$. Then
\begin{eqnarray*}
\chi_{x}^4 &=& 1\\
\chi_{x}\chi_{y^2}  &=& \chi_{y^2}\chi_{x}\chi_t
\end{eqnarray*}
Denote $a=\chi_{x}$, $b=\chi_{y^2}$, $c=\chi_{t}$. Then
$$
\mathbf{G}\left( H^*\right) =\left\langle a, b , c |a^4=b^2=c^2=1, bab=ac\right\rangle\cong G_6.
$$
In addition,
\begin{eqnarray*}
\nu_2 \left(\chi_{y,y^{-1}}\right)&=& \alpha_{1,2}  (y,y^{-1})=1\\
\nu_2 \left(\chi_{x^2y,x^2y^{-1}}\right)&=& \alpha_{1,2} (x^2y,x^2y^{-1})=-1
\end{eqnarray*}
\item Assume that $H\cong H\left(\rightharpoonup_2, \alpha^2 _{1,2}\right)$. Then
\begin{eqnarray*}
\chi_{x}^4 &=& 1\\
\chi_{x}\chi_{y^2}  &=& \chi_{y^2}\chi_{x}
\end{eqnarray*}
Denote $a=\chi_{x}$, $b=\chi_{y^2}$, $c=\chi_{t}$. Then
$$
\mathbf{G}\left( H^*\right) =\left\langle a \right\rangle \times \left\langle b \right\rangle\times \left\langle c \right\rangle\cong \mathbb{Z}_{4} \times\mathbb{Z}_{2} \times \mathbb{Z}_{2}.
$$
In addition,
$$
\nu_2 \left(\chi_{y,y^{-1}}\right)= \nu_2 \left(\chi_{x^2y,x^2y^{-1}}\right)=1
$$
\item Assume that $H\cong H\left(\rightharpoonup_2, \alpha^2 _{1,2} \alpha^2 _{2,2}\right)$. Then
\begin{eqnarray*}
\chi_{x}^4 &=& 1\\
\chi_{x}\chi_{y^2}  &=& \chi_{y^2}\chi_{x}
\end{eqnarray*}
Denote $a=\chi_{x}$, $b=\chi_{y^2}$, $c=\chi_{t}$. Then
$$
\mathbf{G}\left( H^*\right) =\left\langle a \right\rangle \times \left\langle b \right\rangle\times \left\langle c \right\rangle\cong \mathbb{Z}_{4} \times\mathbb{Z}_{2} \times \mathbb{Z}_{2}.
$$
In addition,
$$
\nu_2 \left(\chi_{y,y^{-1}}\right)= \nu_2 \left(\chi_{x^2y,x^2y^{-1}}\right)=-1
$$
\item Assume that $H\cong H\left(\rightharpoonup_2,\beta_1,\alpha _{1,2}\right)\cong H\left(\rightharpoonup_2,\alpha ' \alpha _{1,2}\right)$. Then
\begin{eqnarray*}
\chi_{x}^4 &=& \chi_{t}\\
\chi_{x}\chi_{y^2}  &=& \chi_{y^2}\chi_{x}\chi_{t}=\chi_{y^2}\chi^5_{x}
\end{eqnarray*}
Denote $a=\chi_{x}$ and $b=\chi_{y^2}$. Then
$$
\mathbf{G}\left( H^*\right) =\left\langle a, b  |a^8=b^2=1, bab=a^5\right\rangle\cong G_1.
$$
In addition,
$$
\nu_2 \left(\chi_{y,y^{-1}}\right)= \nu_2 \left(\chi_{x^2y,x^2y^{-1}}\right)=1
$$
\item Assume that $H\cong H\left(\rightharpoonup_2,\beta_1,\alpha _{1,2}\alpha^2_{2,2}\right)\cong H\left(\rightharpoonup_2,\alpha ' \alpha _{1,2}\alpha^2_{2,2}\right)$. Then
\begin{eqnarray*}
\chi_{x}^4 &=& \chi_{t}\\
\chi_{x}\chi_{y^2}  &=& \chi_{y^2}\chi_{x}\chi_{t}=\chi_{y^2}\chi^5_{x}
\end{eqnarray*}
Denote $a=\chi_{x}$ and $b=\chi_{y^2}$. Then
$$
\mathbf{G}\left( H^*\right) =\left\langle a, b  |a^8=b^2=1, bab=a^5\right\rangle\cong G_1.
$$
In addition,
$$
\nu_2 \left(\chi_{y,y^{-1}}\right)= \nu_2 \left(\chi_{x^2y,x^2y^{-1}}\right)=-1
$$
\item Assume that $H\cong H\left(\rightharpoonup_2,\beta_1,\alpha^2_{1,2}\right)\cong H\left(\rightharpoonup_2,\alpha ' \alpha^2_{1,2}\right)$. Then
\begin{eqnarray*}
\chi_{x}^4 &=& \chi_{t}\\
\chi_{x}\chi_{y^2}  &=& \chi_{y^2}\chi_{x}
\end{eqnarray*}
Denote $a=\chi_{x}$ and $b=\chi_{y^2}$. Then
$$
\mathbf{G}\left( H^*\right) =\left\langle a \right\rangle \times \left\langle b \right\rangle\cong \mathbb{Z}_{8} \times\mathbb{Z}_{2} .
$$
In addition,
$$
\nu_2 \left(\chi_{y,y^{-1}}\right)= \nu_2 \left(\chi_{x^2y,x^2y^{-1}}\right)=1
$$
\end{enumerate}
\end{proof}
\begin{proposition}
Assume that $$t\rightharpoonup_3 x = xy^2, \quad t\rightharpoonup_3 y = y.$$ Then  $H\cong H\left( \rightharpoonup_3,   \alpha ^2_{1,2}\right)$ or  $H\left(\rightharpoonup_3,  \beta _1 , \alpha^2 _{1,2}\right)$. Moreover
\begin{enumerate}
\item $\mathbf{G}\left( H\right) \cap Z(H) = \hat G ^L = \left\langle \beta _1 \right\rangle \times \left\langle \beta^2 _2 \right\rangle \cong \mathbb{Z}_{4} \times\mathbb{Z}_{2} $.
\item $\mathbf{G}\left( H^*\right)$ is an abelian group of order $16$ and $H$ has $4$ two-dimensional representations $\pi_1 , \ldots ,\pi _4$  with characters  $\chi_1 , \ldots ,\chi _4$ such that
$$
\nu_2 (\chi _j)=0 \quad\text{ for } j=1,2,3,4
$$

\item When $H\cong H\left(\rightharpoonup_3, \alpha^2 _{1,2}\right)$,
$\mathbf{G}\left( H^*\right) =\left\langle a \right\rangle \times \left\langle b \right\rangle\times \left\langle c \right\rangle\cong \mathbb{Z}_{4} \times\mathbb{Z}_{2} \times \mathbb{Z}_{2}$ and
$$
\pi_j ^2  =  b+ bc +a^2b +a^2bc \quad\text{ for } j=1,2,3,4
$$

\item When $H\cong H\left(\rightharpoonup_3, \beta_1, \alpha^2 _{1,2}\right)$,
$\mathbf{G}\left( H^*\right) =\left\langle a \right\rangle \times \left\langle b \right\rangle\cong \mathbb{Z}_{4} \times\mathbb{Z}_{4}$ and
$$
\pi_j ^2  =  b+ b^3 +a^2b +a^2b^3 \quad\text{ for } j=1,2,3,4
$$
\end{enumerate}

\end{proposition}
\begin{proof}
By Proposition \ref{121}, up to equivalence, there are four Hopf algebras of this type. Two of them,
$H\left( \rightharpoonup_3,   \alpha ^2_{1,2}\right)$ and  $H\left(\rightharpoonup_3,  \beta _1 , \alpha^2 _{1,2}\right)$, are non-trivial.

For further computations we will use the fact that by Proposition \ref{121}  $ H\left(\rightharpoonup_3,\beta_1,\alpha^2_{1,2}\right)$ is equivalent to $ H\left(\rightharpoonup_3,\alpha ' \alpha^2_{1,2}\right)$ for $\alpha '$ defined via
$$
\alpha ' \left( x^{a_{1}}y^{a_{2}},x^{b_{1}}y^{b_{2}}\right) =\left( -1\right)
^{k}
$$
where $k= 0$ if $0\leq a_1+b_1 < 4$ and $k = 1$ if $4\leq a_1+b_1 < 8$.

Let $H$ be a Hopf algebra corresponding to $\rightharpoonup_3$. Then $H$ has $16$ one-dimensional characters of the type $\chi_g \chi_t^j$ for $j=0,1$ and $g\in G^L=\left\langle x^2 \right\rangle \times \left\langle y \right\rangle$ where
\begin{eqnarray*}
\chi_g (p_h) &=& \delta _{g,h} \qquad \chi_g (t) =1\\
\chi_t (p_h) &=& \delta _{1,h} \qquad \chi_g (t) =-1
\end{eqnarray*}
Then $\mathbf{G}\left( H^*\right)$ is abelian and
$$
\chi_{y}^4  =\chi_t^2=\chi_1 =1
$$

$H$ has $4$ two-dimensional characters $\chi_{g,t\rightharpoonup_3 g}$ for $g\in \{ x, x^{-1}, xy,x^{-1}y\}$ where
$$
\chi_{g,t\rightharpoonup_3 g} (p_h) = \delta _{g,h}+ \delta _{t\rightharpoonup_3 g,h} \qquad 
\chi_{g,t\rightharpoonup_3 g} (p_h t) =0.
$$
Then for all $g\in \{ x, x^{-1}, xy,x^{-1}y\}$
\begin{eqnarray*}
\chi_{g, t\rightharpoonup_3 g } ^2 &=& \chi _{x^2}+ \chi _{x^2}\chi _t +\chi _{x^2y^2} +\chi _{x^2y^2}\chi _t\\
\nu_2 \left(\chi_{g, t\rightharpoonup_3 g }\right)&=&0
\end{eqnarray*}

If $H\cong H\left(\rightharpoonup_3, \alpha^2 _{1,2}\right)$ then
$\chi_{x^2}^2 = 1$ and
$$
\mathbf{G}\left( H^*\right) =\left\langle a \right\rangle \times \left\langle b \right\rangle\times \left\langle c \right\rangle\cong \mathbb{Z}_{4} \times\mathbb{Z}_{2} \times \mathbb{Z}_{2}
$$
for $a=\chi_{y}$, $b=\chi_{x^2}$, and $c=\chi_{t}$.

If $H\cong H\left(\rightharpoonup_3,\beta_1,\alpha^2_{1,2}\right)\cong H\left(\rightharpoonup_3,\alpha ' \alpha^2_{1,2}\right)$ then
$\chi_{x^2}^2 = \chi_{t}$ and
$$
\mathbf{G}\left( H^*\right) =\left\langle a \right\rangle \times \left\langle b \right\rangle\cong \mathbb{Z}_{4} \times\mathbb{Z}_{4}
$$
for $a=\chi_{y}$ and $b=\chi_{x^2}$.
\end{proof}

\begin{proposition}
Assume that $$t\rightharpoonup_4 x = x^{-1}y^2, \quad t\rightharpoonup_4 y = y^{-1}.$$ Then  $H\cong H\left( \rightharpoonup_4,   \alpha ^2_{1,2}\right)$ or  $H\left(\rightharpoonup_4,   \alpha^2 _{1,2} \alpha^2 _{2,2}\right)$. Moreover
\begin{enumerate}
\item $\mathbf{G}\left( H\right) \cap Z(H) = \hat G ^L = \left\langle \beta^2 _1 \right\rangle \times \left\langle \beta^2 _2 \right\rangle \cong \mathbb{Z}_{2} \times\mathbb{Z}_{2} $.
\item $\mathbf{G}\left( H^*\right) =\left\langle a \right\rangle \times \left\langle b \right\rangle\times \left\langle c \right\rangle\cong \mathbb{Z}_{2} \times\mathbb{Z}_{2} \times \mathbb{Z}_{2}$.
\item $H$ has $6$ two-dimensional representations $\pi_1 , \ldots ,\pi _6$ with characters  $\chi_1 , \ldots ,\chi _6$ such that
\begin{eqnarray*}
\pi_1 ^2 &=&\pi_2 ^2 = 1+ c +b +bc\\
\pi_3 ^2 &=&\pi_4 ^2 = a+ ac +b +bc\\
\pi_5 ^2 &=&\pi_6 ^2 = ab +abc+b +bc\\
\nu_2 (\chi _j)&=&0 \quad \text{ for } j=3,4,5,6
\end{eqnarray*}
\item When $H\cong H\left(\rightharpoonup_4, \alpha^2 _{1,2}\right)$,
$$
\nu_2 (\chi _j)=1\quad \text{ for } j=1,2
$$
\item When $H\cong H\left(\rightharpoonup_4, \alpha^2_{1,2}\alpha^2_{2,2}\right)$,
$$
\nu_2 (\chi _j)=-1\quad \text{for }  j=1,2
$$
\end{enumerate}
\end{proposition}
\begin{proof}
By Proposition \ref{-12-1}, up to equivalence, there are four Hopf algebras of this type. Two of them,
$H\left( \rightharpoonup_4,   \alpha ^2_{1,2}\right)$ and $H\left(\rightharpoonup_4,  \beta _1 , \alpha^2 _{1,2}\right)$, are non-trivial.

Let $H$ be a Hopf algebra corresponding to $\rightharpoonup_4$. Then $H$ has $8$ one-dimensional characters of the type $\chi_g \chi_t^j$ for $j=0,1$ and $g\in G^L=\left\langle x^2 \right\rangle \times \left\langle y^2 \right\rangle$ where
\begin{eqnarray*}
\chi_g (p_h) &=& \delta _{g,h} \qquad \chi_g (t) =1\\
\chi_t (p_h) &=& \delta _{1,h} \qquad \chi_g (t) =-1
\end{eqnarray*}
For each one-dimensional character $\chi$, $\chi^2 =1$. Set $a=\chi _{x^2}$, $b=\chi _{y^2}$ and $c=\chi _{t}$. Then
$$
\mathbf{G}\left( H^*\right) =\left\langle a \right\rangle \times \left\langle b \right\rangle\times \left\langle c \right\rangle\cong \mathbb{Z}_{2} \times\mathbb{Z}_{2} \times \mathbb{Z}_{2}.
$$
$H$ has $6$ two-dimensional characters $\chi_{g,t\rightharpoonup_4 g}$ for $g\in \{ y,x^2y, x, x^{-1}, xy, xy^{-1}\}$ where
$$
\chi_{g,t\rightharpoonup_4 g} (p_h) = \delta _{g,h}+ \delta _{t\rightharpoonup_4 g,h} \qquad \chi_{g,t\rightharpoonup_4 g} (p_h t) =0.
$$
Then
\begin{eqnarray*}
\chi_{y,y^{-1}} ^2 &=&\chi_{x^2y,x^2y^{-1}} ^2 = 1+ \chi _t +\chi _{y^2} +\chi _{y^2}\chi _t\\
\chi_{x,x^{-1}y^2} ^2 &=&\chi_{x^{-1},xy^2} ^2 = \chi _{x^2} +\chi _{x^2}\chi _t+\chi _{y^2} +\chi _{y^2}\chi _t\\
\chi_{xy,x^{-1}y} ^2 &=&\chi_{xy^{-1},x^{-1}y^{-1}} ^2 = \chi _{x^2y^2} +\chi _{x^2y^2}\chi _t+\chi _{y^2} +\chi _{y^2}\chi _t\\
\nu_2 \left(\chi_{g, t\rightharpoonup_4 g }\right)&=&0 \quad \text{ for } g\in \{ x, x^{-1}, xy, xy^{-1}\}
\end{eqnarray*}

If $H\cong H\left(\rightharpoonup_4, \alpha^2 _{1,2}\right)$ then  by formula (\ref{nu_2})
\begin{eqnarray*}
\nu_2 \left(\chi_{y, y ^{-1}}\right)&=&\alpha^2 _{1,2}(y, y ^{-1})=1\\
\nu_2 \left(\chi_{x^2y, x^2y ^{-1}}\right)&=&\alpha^2 _{1,2}(x^2y, x^2y ^{-1})=1
\end{eqnarray*}

If $H\cong H\left(\rightharpoonup_4, \alpha^2 _{1,2}\alpha^2 _{2,2}\right)$ then  by formula (\ref{nu_2})
\begin{eqnarray*}
\nu_2 \left(\chi_{y, y ^{-1}}\right)&=&\alpha^2 _{1,2}(y, y ^{-1})\alpha^2 _{2,2}(y, y ^{-1})=-1\\
\nu_2 \left(\chi_{x^2y, x^2y ^{-1}}\right)&=&\alpha^2 _{1,2}(x^2y, x^2y ^{-1})\alpha^2 _{2,2}(x^2y, x^2y ^{-1})=-1
\end{eqnarray*}
\end{proof}


\begin{proposition}
Assume that $$t\rightharpoonup_5 x = y, \quad t\rightharpoonup_5 y = x.$$ Then  $H\cong H\left( \rightharpoonup_5,   \alpha^{\prime} _{1,2}\right)$ or  $H\left(\rightharpoonup_5,  \left(\alpha^{\prime} _{1,2} \right)^2\right)$ for $\alpha^{\prime} _{1,2}$  defined via
$$
\alpha _{1,2}^{\prime }\left( x^{a_{1}}y^{a_{2}},x^{b_{1}}y^{b_{2}}\right)
= \omega^{a_{1}b_{2}-a_{2}b_{1}}(-1)^{(a_1+b_1)k_2+(a_2+b_2)k_1}
$$
where $\omega$ is a primitive $8$-th root of $1$ such that $\omega^2 =i$, and $k_j = 0$ if $0\leq a_j+b_j <4$ and $k_j = 1$ if $4\leq a_j+b_j <8$. Moreover
\begin{enumerate}
\item $\mathbf{G}\left( H\right) \cap Z(H) = \hat G ^L = \left\langle \beta _1  \beta _2 \right\rangle \cong \mathbb{Z}_{4}  $.
\item $\mathbf{G}\left( H^*\right) =\left\langle a \right\rangle \times \left\langle b \right\rangle\cong \mathbb{Z}_{4} \times\mathbb{Z}_{2} $.
\item $H$ has $6$ two-dimensional representations $\pi_1 , \ldots ,\pi _6$ with characters  $\chi_1 , \ldots ,\chi _6$ such that
\begin{eqnarray*}
\pi_1 ^2 &=&\pi_2 ^2 =a+ ab +\pi_6\\
\pi_3 ^2 &=&\pi_4 ^2 = a^3+ a^3b +\pi_6\\
\pi_5 ^2 &=&\pi_6 ^2 = 1+ b +a^2 +a^2b\\
\nu_2 (\chi _j)&=&0 \quad \text{ for } j=1,2,3,4\\
\nu_2 (\chi _j)&=&1\quad \text{ for } j=5,6\\
\nu_4 (\chi _j)&=&1 \quad \text{ for } j=1,2,3,4\\
\nu_4 (\chi _j)&=&2\quad \text{ for } j=5,6
\end{eqnarray*}
\item When $H\cong H\left(\rightharpoonup_5,  \alpha^{\prime} _{1,2}\right)$,
\begin{eqnarray*}
\nu_8 (\chi _j)&=&0\quad \text{ for } j=1,2,3,4\\
\nu_8 (\chi _j)&=&2\quad \text{ for } j=5,6
\end{eqnarray*}
\item When $H\cong H\left(\rightharpoonup_5, \left(\alpha^{\prime} _{1,2} \right)^2\right)$,
$$
\nu_8 (\chi _j)=2\quad \text{for }  j=1,2,3,4,5,6
$$
\end{enumerate}
\end{proposition}
\begin{proof}
By Proposition \ref{xy}, up to equivalence, there are four Hopf algebras of this type. Three of them,
$H\left( \rightharpoonup_5,  \alpha^{\prime} _{1,2}\right)$,  $H\left(\rightharpoonup_5, \left(\alpha^{\prime} _{1,2} \right)^2\right)$, and $H\left(\rightharpoonup_5, \left(\alpha^{\prime} _{1,2} \right)^3\right)$, are non-trivial.

By Proposition \ref{xy}, $H\left(\rightharpoonup_5, \left(\alpha^{\prime} _{1,2} \right)^j\right)$ is equivalent to $H\left( \rightharpoonup_5, \beta^j,  \alpha^j _{1,2}\right)$ where
$$
\beta\left( x^{a_{1}}y^{a_{2}}\right) =i^{a_1 a_2}.
$$

Define
\begin{eqnarray*}
f :G &\to & G \text{ via} \\
f\left( x\right) &=&t\rightharpoonup_5 x =y\\
f\left( y\right) &=&t\rightharpoonup_5 y =x
\end{eqnarray*}
Then $f ( t\rightharpoonup_5 g) =t\rightharpoonup_5 f (  g) $  and
\begin{eqnarray*}
\beta \circ f \left( x^{a_{1}}y^{a_{2}}\right) &=&  \beta\left( x^{a_{2}}y^{a_{1}}\right)
=i^{ a_{2}a_{1}}\\
\alpha _{1,2}\circ (f \times f) \left( x^{a_{1}}y^{a_{2}},x^{b_{1}}y^{b_{2}}\right) &=&  \alpha _{1,2}\left( x^{a_{2}}y^{a_{1}},x^{b_{2}}y^{b_{1}}\right)
=i^{ a_{2}b_{1}}
\end{eqnarray*}
and therefore
 \begin{eqnarray*}
\beta \circ f  &=&  \beta\\
\alpha _{1,2}\circ (f \times f) &=&  \alpha _{2,1}
\end{eqnarray*}
Define $\mu
\in \left( k^{G}\right) ^{\times }$  via
$$
\mu \left( x^{a_{1}}y^{a_{2}}\right)
=i^{a_{1}a_{2}}
$$
Then, as in part 1 of Proposition \ref{12_21},
$\alpha _{2,1}\left( g,h\right)
\mu \left( g\right) \mu \left( h\right) \mu \left( gh \right)
^{-1} =\alpha ^{3}_{1,2}\left( g,h\right)$. Moreover
$$
\beta (g) \mu ^{-1}\left( g\right) \mu ^{-1}\left( t\rightharpoonup_5 g\right) =\beta ^{3}(g)
$$

Thus $[(\sigma^{\beta } , \tau _{2,1})] = [(\sigma^{\beta ^{3}} , \tau ^{3}_{1,2})] $ and therefore
$$
H(\rightharpoonup_5 , \alpha^{\prime } _{1,2}) \cong H(\rightharpoonup_5 , \beta, \alpha _{1,2})\cong H(\rightharpoonup_5 , \beta, \alpha _{2,1}) \cong H(\rightharpoonup_5 , \beta^3, \alpha^3 _{1,2}) \cong H( \rightharpoonup_5 , (\alpha^{\prime } _{1,2})^{3}).
$$

Let $H$ be a Hopf algebra corresponding to $\rightharpoonup_5$. Then $H$ has $8$ one-dimensional characters of the type $\chi_g \chi_t^j$ for $j=0,1$ and $g\in G^L=\left\langle x y \right\rangle$ where
\begin{eqnarray*}
\chi_g (p_h) &=& \delta _{g,h} \qquad \chi_g (t) =1\\
\chi_t (p_h) &=& \delta _{1,h} \qquad \chi_g (t) =-1
\end{eqnarray*}
Then
\begin{eqnarray*}
\chi_{xy}^4 &=&  \chi_t ^2 =\chi_1 =1\\
\chi_{xy}\chi_t &=&  \chi_t\chi_{xy}
\end{eqnarray*}
Set $a=\chi _{xy}$, $b=\chi _{t}$. Then
$$
\mathbf{G}\left( H^*\right) =\left\langle a \right\rangle \times \left\langle b \right\rangle\cong \mathbb{Z}_{4} \times\mathbb{Z}_{2} .
$$
$H$ has $6$ two-dimensional characters $\chi_{g,t\rightharpoonup_5 g}$ for $g\in \{ x,  x^{-1},  xy^{2}, x^{-1}y^2,  xy^{-1}, x^2\}$ where
$$
\chi_{g,t\rightharpoonup_5 g} (p_h) = \delta _{g,h}+ \delta _{t\rightharpoonup_5 g,h} \qquad \chi_{g,t\rightharpoonup_5 g} (p_h t) =0.
$$
Then
\begin{eqnarray*}
\chi_{x,y} ^2 &=&\chi_{x^{-1}y^2,x^2y^{-1}} ^2 =\chi _{xy} +\chi _{xy}\chi _t+\chi _{x^2,y^2}\\
\chi_{x^{-1},y^{-1}} ^2&=&\chi_{xy^2,x^2y} ^2 =\chi _{x^{-1}y^{-1}} +\chi _{x^{-1}y^{-1}}\chi _t+\chi _{x^2,y^2}\\
\chi_{xy^{-1},x^{-1}y} ^2 &=&\chi_{x^{2},y^{2}} ^2 = 1+ \chi _t+\chi _{x^2y^2} +\chi _{x^2y^2}\chi _t\\
\nu_2 \left(\chi_{x^{2},y^{2}}\right) &=&1\\
\nu_2 \left(\chi_{xy^{-1},x^{-1}y}\right) &=& \tau_t(xy^{-1},x^{-1}y)=1\\
\nu_2 \left(\chi_{g, t\rightharpoonup_5 g }\right)&=&0 \quad \text{ for } g\in \{x,  x^{-1}, xy^{2}, x^{-1}y^2\}
\end{eqnarray*}

We will now compute $4$-th and $8$-th Frobenius-Schur indicators for the two-dimensional characters. As in\cite[Section 4]{KMM}, $\Lambda = \frac{1}{2}(p_1 +p_1t)$ is the integral of $H$ such that $\varepsilon (\Lambda ) =1$. Then, since $\tau _t = \alpha^{\prime } _{1,2}$ or $(\alpha^{\prime } _{1,2})^2$ , $h(t\rightharpoonup_5 h) =(xy)^j$, and $h^4=1$ for all $h\in G$,
\begin{eqnarray*}
\Delta (\Lambda ) &=&\frac{1}{2}\left(\sum_{h_1h_2 =1}p_{h_1}\otimes p_{h_2}  +\sum_{h_1h_2 =1}\tau_t (h_1, h_2) p_{h_1}t\otimes p_{h_2}t\right)\\
\Delta_4 (\Lambda ) &=&\frac{1}{2}\left(\sum_{h_1h_2h_3h_4 =1}p_{h_1}\otimes p_{h_2}\otimes p_{h_3}\otimes p_{h_4} \right. \\
&&\left. +\sum_{h_1h_2h_3h_4 =1}\tau_t (h_1 h_2,h_3h_4)
\tau_t (h_1, h_2)\tau_t (h_3,h_4) p_{h_1}t\otimes p_{h_2}t\otimes p_{h_3}t\otimes p_{h_4}t\right)\\
\Lambda ^{[4]} &=&\frac{1}{2}\left(\sum_{h^4 =1}p_{h}  +\sum_{\left(h(t\rightharpoonup_5 h)\right)^2 =1}\tau_t (h(t\rightharpoonup_5 h),h(t\rightharpoonup_5 h))
\left(\tau_t (h,t\rightharpoonup_5 h)\right)^2 p_{h}\right)\\
 &=&\frac{1}{2}\left(1  +\sum_{h\in \left\langle x y \right\rangle\times\left\langle x^2 \right\rangle }
\left(\tau_t (h,t\rightharpoonup_5 h)\right)^2 p_{h}\right)\\
\Delta_8 (\Lambda ) &=&\frac{1}{2}\left(\sum_{h_1\cdots h_8 =1}p_{h_1}\otimes\cdots\otimes p_{h_8}  +\sum_{h_1\cdots h_8 =1}\tau_t (h_1 h_2h_3h_4,h_5h_6h_7h_8)\tau_t (h_1 h_2,h_3h_4)\tau_t (h_5h_6,h_7h_8)\right. \\
&&\left.
\phantom{\sum_{h_1\cdots h_8 =1}}\cdot\ \tau_t (h_1, h_2)\tau_t (h_3,h_4) \tau_t (h_5,h_6)\tau_t (h_7,h_8) p_{h_1}t\otimes \cdots\otimes p_{h_8}t\right)\\
\Lambda ^{[8]} &=&\frac{1}{2}\left(\sum_{h^8 =1}p_{h}  +\sum_{\left(h(t\rightharpoonup_5 h)\right)^4 =1}\tau_t (h(t\rightharpoonup_5 h)^2,h(t\rightharpoonup_5 h)^2)\left(\tau_t (h(t\rightharpoonup_5 h),h(t\rightharpoonup_5 h))\right)^2
\right. \\
&&\left. \phantom{\sum_{h^8 =1}} \cdot \ \left(\tau_t (h,t\rightharpoonup_5 h)\right)^4 p_{h}\right)=\frac{1}{2}\left(1  +\sum_{h\in G}
\left(\tau_t (h,t\rightharpoonup_5 h)\right)^4 p_{h}\right)
\end{eqnarray*}
Then
\begin{eqnarray*}
\nu_4 \left(\chi_{g, t\rightharpoonup_5 g }\right)&=&1 \quad \text{ for } g\in \{x,  x^{-1}, xy^{2}, x^{-1}y^2\}\\
\nu_4 \left(\chi_{xy^{-1},x^{-1}y}\right) &=& \frac{1}{2}\left(2  + \left(\tau_t(xy^{3},x^{3}y)\right)^2 + \left(\tau_t(x^{3}y,xy^{3})\right)^2\right)=\frac{1}{2}\left(2  + 1 +1\right) =2\\
\nu_4 \left(\chi_{x^2,y^2}\right) &=& \frac{1}{2}\left(2  + \left(\tau_t(x^2,y^2)\right)^2 + \left(\tau_t(y^2,x^{2})\right)^2\right)=\frac{1}{2}\left(2  + 1 +1\right) =2
\end{eqnarray*}

If $H\cong H\left(\rightharpoonup_5, \left(\alpha^{\prime} _{1,2} \right)^2\right)$ then $\Lambda ^{[8]} =1$ and
$$
\nu_8 \left(\chi_{g, t\rightharpoonup_5 g }\right)=2 \quad \text{ for all } g\in \{ x,  x^{-1},  xy^{2}, x^{-1}y^2,  xy^{-1}, x^2\}
$$

If $H\cong H\left( \rightharpoonup_5,  \alpha^{\prime} _{1,2}\right)$ then
\begin{eqnarray*}
\nu_8\left(\chi_{xy^{-1},x^{-1}y}\right) &=&
\nu_8 \left(\chi_{x^2,y^2}\right)=\frac{1}{2}\left(2  + 1 +1\right) =2\\
\nu_8 \left(\chi_{g, t\rightharpoonup_5 g }\right)&= &\frac{1}{2}\left(2  + (-1) +(-1)\right) =0\quad \text{ for  } g\in \{x,  x^{-1},  xy^{2}, x^{-1}y^2\}.
\end{eqnarray*}
Moreover, note that ${\rm exp}\left(H\left(\rightharpoonup_5, \left(\alpha^{\prime} _{1,2} \right)^2\right)\right) =8$ and ${\rm exp}\left(H\left( \rightharpoonup_5,  \alpha^{\prime} _{1,2}\right)\right) =16$.

\end{proof}


\begin{proposition}
Assume that $$t\rightharpoonup_6 x = xy^2, \quad t\rightharpoonup_6 y = x^2y^{-1}.$$ Then  $H\cong H\left( \rightharpoonup_6,   \alpha _{1,1}  \alpha _{1,2} \alpha _{2,2}\right)$ or  $H\left(\rightharpoonup_6, \alpha^2 _{1,1}  \alpha^2 _{1,2} \alpha^2 _{2,2}\right)$. Moreover
\begin{enumerate}
\item $\mathbf{G}\left( H\right) \cap Z(H) = \hat G ^L = \left\langle \beta^2 _1 \right\rangle \times \left\langle \beta^2 _2 \right\rangle \cong \mathbb{Z}_{2} \times\mathbb{Z}_{2} $.
\item $\mathbf{G}\left( H^*\right) =\left\langle a \right\rangle \times \left\langle b \right\rangle\times \left\langle c \right\rangle\cong \mathbb{Z}_{2} \times\mathbb{Z}_{2} \times \mathbb{Z}_{2}$.
\item $H$ has $6$ two-dimensional representations $\pi_1 , \ldots ,\pi _6$ with characters  $\chi_1 , \ldots ,\chi _6$ such that
\begin{eqnarray*}
\pi_1 ^2 &=&\pi_2 ^2 = a+ ac +ab +abc\\
\pi_3 ^2 &=&\pi_4 ^2 = a+ ac +b +bc\\
\pi_5 ^2 &=&\pi_6 ^2 = ab +abc+b +bc\\
\nu_2 (\chi _j)&=&0 \quad \text{ for } j=1,2,3,4,5,6
\end{eqnarray*}
\item When $H\cong H\left(\rightharpoonup_6,   \alpha _{1,1}  \alpha _{1,2} \alpha _{2,2}\right)$,
$$
\nu_4 (\chi _j)=0\quad \text{ for } j=1,2,3,4,5,6
$$
\item When $H\cong H\left(\rightharpoonup_6, \alpha^2 _{1,1}  \alpha^2 _{1,2} \alpha^2 _{2,2}\right)$,
$$
\nu_4 (\chi _j)=2\quad \text{for }  j=1,2,3,4,5,6
$$
\end{enumerate}
\end{proposition}
\begin{proof}
By Proposition \ref{122-1}, up to equivalence, there are four Hopf algebras of this type. Three of them,
$H\left( \rightharpoonup_6,   \alpha _{1,1}  \alpha _{1,2} \alpha _{2,2}\right)$,  $H\left(\rightharpoonup_6, \alpha^2 _{1,1}  \alpha^2 _{1,2} \alpha^2 _{2,2}\right)$, and $H\left(\rightharpoonup_6, \alpha^3 _{1,1}  \alpha^3 _{1,2} \alpha^3 _{2,2}\right)$, are non-trivial.

Define
\begin{eqnarray*}
f :G &\to & G \text{ via} \\
f\left( x\right) &=&t\rightharpoonup_6 x =xy^2\\
f\left( y\right) &=&t\rightharpoonup_6 y =x^2y^{-1}
\end{eqnarray*}
Then $f ( t\rightharpoonup_6 g) =t\rightharpoonup_6 f (  g) $  and
\begin{eqnarray*}
\alpha _{1,1}\alpha _{1,2}\alpha  _{2,2}\circ (f \times f) \left( x^{a_{1}}y^{a_{2}},x^{b_{1}}y^{b_{2}}\right) &=&  \alpha _{1,1}\alpha _{1,2}\alpha  _{2,2}\left( x^{a_{1}+2a_{2}}y^{2a_{1}-a_{2}},x^{b_{1}+2b_{2}}y^{2b_{1}-b_{2}}\right)\\
&=&i^{ a_{1}b_{1}}(-1)^{a_{1}b_{2}+a_{2}b_{1}}  i^{ -a_{1}b_{2}}(-1)^{a_{1}b_{1}+a_{2}b_{2}}  i^{ a_{2}b_{2}}(-1)^{a_{1}b_{2}+a_{2}b_{1}} \\ &=&i^{-a_{1}b_{1}-a_{1}b_{2}-a_{2}b_{2}}
\end{eqnarray*}
and therefore
\begin{equation*}
\alpha _{1,1}\alpha _{1,2}\alpha  _{2,2}\circ (f \times f)= \alpha ^{3}_{1,1}\alpha ^{3}_{1,2}\alpha ^{3} _{2,2}
\end{equation*}
Thus $H(\rightharpoonup_6, \alpha _{1,1}\alpha _{1,2}\alpha  _{2,2})\cong  H(\rightharpoonup_6  \alpha ^{3}_{1,1}\alpha ^{3}_{1,2}\alpha ^{3} _{2,2})$.

Let $H$ be a Hopf algebra corresponding to $\rightharpoonup_6$. Then $H$ has $8$ one-dimensional characters of the type $\chi_g \chi_t^j$ for $j=0,1$ and $g\in G^L=\left\langle x^2 \right\rangle \times \left\langle y^2 \right\rangle$ where
\begin{eqnarray*}
\chi_g (p_h) &=& \delta _{g,h} \qquad \chi_g (t) =1\\
\chi_t (p_h) &=& \delta _{1,h} \qquad \chi_g (t) =-1
\end{eqnarray*}
For each one-dimensional character $\chi$, $\chi^2 =1$. Set $a=\chi _{x^2}$, $b=\chi _{y^2}$ and $c=\chi _{t}$. Then
$$
\mathbf{G}\left( H^*\right) =\left\langle a \right\rangle \times \left\langle b \right\rangle\times \left\langle c \right\rangle\cong \mathbb{Z}_{2} \times\mathbb{Z}_{2} \times \mathbb{Z}_{2}.
$$
$H$ has $6$ two-dimensional characters $\chi_{g,t\rightharpoonup_6 g}$ for $g\in \{ x,  x^{-1}, y, y^{-1}, xy, xy^{-1}\}$ where
$$
\chi_{g,t\rightharpoonup_6 g} (p_h) = \delta _{g,h}+ \delta _{t\rightharpoonup_6 g,h} \qquad \chi_{g,t\rightharpoonup_6 g} (p_h t) =0.
$$
Then
\begin{eqnarray*}
\chi_{x,xy^2} ^2 &=&\chi_{x^{-1},x^{-1}y^2} ^2 = \chi _{x^2} +\chi _{x^2}\chi _t+\chi _{x^2y^2} +\chi _{x^2y^2}\chi _t\\
\chi_{y,x^2y^{-1}} ^2 &=&\chi_{y^{-1},x^2y} ^2 = \chi _{x^2}+ \chi _{x^2}\chi _t +\chi _{y^2} +\chi _{y^2}\chi _t\\
\chi_{xy,x^{-1}y} ^2 &=&\chi_{xy^{-1},x^{-1}y^{-1}} ^2 = \chi _{x^2y^2} +\chi _{x^2y^2}\chi _t+\chi _{y^2} +\chi _{y^2}\chi _t\\
\nu_2 \left(\chi_{g, t\rightharpoonup_6 g }\right)&=&0 \quad \text{ for all } g\in \{x,  x^{-1}, y, y^{1}, xy, xy^{-1}\}
\end{eqnarray*}

We will now compute $4$-th Frobenius-Schur indicators for the two-dimensional characters. Recall that $\Lambda = \frac{1}{2}(p_1 +p_1t)$ is the integral of $H$ such that $\varepsilon (\Lambda ) =1$. Then, since $\tau _t = \alpha _{1,1}  \alpha _{1,2} \alpha _{2,2}$ or $\alpha ^2_{1,1}  \alpha ^2_{1,2} \alpha ^2_{2,2}$ , $h(t\rightharpoonup_6 h) =x^{2j}y^{2k}$, and $h^4=1$ for all $h\in G$,
\begin{eqnarray*}
\Delta (\Lambda ) &=&\frac{1}{2}\left(\sum_{h_1h_2 =1}p_{h_1}\otimes p_{h_2}  +\sum_{h_1h_2 =1}\tau_t (h_1, h_2) p_{h_1}t\otimes p_{h_2}t\right)\\
\Delta_4 (\Lambda ) &=&\frac{1}{2}\left(\sum_{h_1h_2h_3h_4 =1}p_{h_1}\otimes p_{h_2}\otimes p_{h_3}\otimes p_{h_4} \right. \\
&&\left. +\sum_{h_1h_2h_3h_4 =1}\tau_t (h_1 h_2,h_3h_4)
\tau_t (h_1, h_2)\tau_t (h_3,h_4) p_{h_1}t\otimes p_{h_2}t\otimes p_{h_3}t\otimes p_{h_4}t\right)\\
\Lambda ^{[4]} &=&\frac{1}{2}\left(\sum_{h^4 =1}p_{h}  +\sum_{\left(h(t\rightharpoonup_6 h)\right)^2 =1}\tau_t (h(t\rightharpoonup_6 h),h(t\rightharpoonup_6 h))
\left(\tau_t (h,t\rightharpoonup_6 h)\right)^2 p_{h}\right)\\
 &=&\frac{1}{2}\left(1  +\sum_{h\in G}
\left(\tau_t (h,t\rightharpoonup_6 h)\right)^2 p_{h}\right)
\end{eqnarray*}

If $H\cong H\left(\rightharpoonup_6, \alpha ^2_{1,1}  \alpha ^2_{1,2} \alpha ^2_{2,2}\right)$ then $\Lambda ^{[4]} =1$ and
$$
\nu_4 \left(\chi_{g, t\rightharpoonup_6 g }\right)=2 \quad \text{ for all } g\in \{x,  x^{-1}, y, y^{-1}, xy, xy^{-1}\}.
$$

If $H\cong H\left(\rightharpoonup_6, \alpha _{1,1}  \alpha _{1,2} \alpha _{2,2}\right)$ then
$$
\nu_4 \left(\chi_{g, t\rightharpoonup_6 g }\right)= \frac{1}{2}\left(2  + (-1) +(-1)\right) =0\quad \text{ for all } g\in \{x,  x^{-1}, y, y^{-1}, xy, xy^{-1}\}.
$$
Moreover, note that ${\rm exp}\left(H\left(\rightharpoonup_6, \alpha ^2_{1,1}  \alpha ^2_{1,2} \alpha ^2_{2,2}\right)\right) =4$ and ${\rm exp}\left(H\left(\rightharpoonup_6, \alpha _{1,1}  \alpha _{1,2} \alpha _{2,2}\right)\right) =8$.
\end{proof}
Since Hopf algebras associated with two different actions are isomorphic only if the corresponding matrices are similar, the above six propositions can be combined in the following theorem:
\begin{theorem}\label{iso4x4}
There are exactly $16$ nontrivial non-isomorphic semisimple
Hopf algebras of dimension $32$ with group of group-like elements isomorphic to $%
\mathbb{Z}_{4}\times \mathbb{Z}_{4}$. Moreover, all of them can be distinguished by categorical invariants, such as their Grothendieck ring structure or their (higher) Frobenius-Schur indicators. Therefore, the corresponding representation categories are also different and these Hopf algebras are not twist-equivalent to each other.
\end{theorem}


\end{document}